\documentclass[11pt]{article}
\usepackage[a4paper, margin=20mm]{geometry}

\usepackage{hyperref}

\usepackage{amssymb}
\usepackage{amsmath}
 \usepackage{amsthm}
 \usepackage{cleveref}
\usepackage{tikz}
\usetikzlibrary{arrows.meta}
\usetikzlibrary{arrows.meta}
\usetikzlibrary{bending}
\usetikzlibrary{
  decorations.pathreplacing,
  decorations.markings,
  patterns
}

 \theoremstyle{definition}
\newtheorem{theorem}{Theorem}[section]
\newtheorem{prop}[theorem]{Proposition}
\newtheorem{lemma}[theorem]{Lemma}
\newtheorem{cor}[theorem]{Corollary}
\newtheorem{rem}[theorem]{Remark}
\newtheorem{definition}[theorem]{Definition}
\newtheorem{example}[theorem]{Example}

\newcounter{res}

\newtheorem{result}[res]{Theorem}   				

\numberwithin{equation}{section}

\newcommand{\Lip}{\mathrm{Lip}}
\DeclareMathOperator{\diam}{diam}

\DeclareMathOperator{\Con}{Con}
\DeclareMathOperator{\pCon}{pCon}
\DeclareMathOperator*{\NN}{\mathbb{N}}
\DeclareMathOperator*{\ZZ}{\mathbb{Z}}
\DeclareMathOperator*{\RR}{\mathbb{R}}
\usepackage{tikz-cd}
\usetikzlibrary{decorations.pathreplacing}
\usepackage{enumitem}

 \usepackage{lineno}

\makeatletter
\renewcommand{\@fnsymbol}[1]{%
  \ifcase#1\relax 
  \or \dag 
  \or \ddag 
  \or \S 
  \or \P 
  \or \| 
  \else\@ctrerr\fi}
\makeatother

\title{On the topology of limit sets\\ of non-autonomous iterated function systems}
\author{Yuto Nakajima\\
Faculty of Science and Engineering, Doshisha University,\\ 1-3 Tatara Miyakodani, Kyotanabe-shi, Kyoto, 610-0394, Japan.\\
E-mail: yunakaji@mail.doshisha.ac.jp
\\
ORCID: 0000-0002-0357-4160
\and
Takayuki Watanabe\footnote{Author to whom any correspondence should be addressed.}
\\College of Science and Engineering, Chubu University,
\\
1200 Matsumoto-cho, Kasugai-shi, Aichi, 487-8501, Japan. \\
E-mail: takawatanabe@fsc.chubu.ac.jp\\
ORCID: 0009-0000-3591-7351
}

\begin{document}
\maketitle
\begin{abstract}
Since Mandelbrot's seminal work, 
there has been growing interest in the geometric nature of fractals. 
While the topological properties of the limit sets of
IFSs have been studied---notably in the pioneering work of Hata---many aspects remain poorly understood, especially in the non-autonomous setting. 
In this paper, we investigate the topology of limit sets arising from randomly generated non-autonomous IFSs. To this end, we develop a simplicial-homological framework that makes their topological structure accessible to rigorous analysis.
We apply our abstract theory to the concrete analysis of the so-called fractal squares,  
and provide an answer to a variant of Mandelbrot’s percolation problem. 
Moreover, for the non-autonomous fractal squares considered here, we prove that the 
Betti numbers of the finite-stage approximations grow exponentially  
at a rate equal to the natural symbolic entropy of the system. 
This reveals a quantitative link between topology across scales and dynamical complexity.
\end{abstract}




\section{Introduction and the main theorems}\label{sec:Intro}
\subsection{Background}
Fractals are ubiquitous in nature. 
The shapes of coastlines, clouds, and forests are typical examples. 
Mandelbrot pointed out in his seminal work \cite{Mandel} 
that these complicated shapes exhibit self-similarity, sometimes only in a statistical sense. 
The fractals which we study in the present paper are limit sets of non-autonomous iterated function systems, defined below. 

\begin{definition}\label{def:NIFS}
Let $X$ be a compact metric space. 
For a map $f \colon X \to X$,  we denote its  Lipschitz constant by $\Lip (f)$. 

A non-autonomous iterated function system $(\Phi^{(j)})_{j=1}^{\infty}$ on $X$ is a sequence of collections   $\Phi^{(j)} = \{f_{i}^{(j)}  \colon X \to X\}_{i \in I^{(j)}}$ of maps, where each index set $I^{(j)}$ is finite, and 
there exists a uniform constant $c < 1$ such that $\Lip (f_{i}^{(j)}) \leq c$ for all $j \geq 1$ and $i \in I^{(j)}$. 

For a non-autonomous IFS $(\Phi^{(j)})_{j=1}^{\infty}$, 
we endow $I^{(j)}$ with the discrete topology and endow $\prod_{j=1}^{\infty} I^{(j)}$ with the product topology.
Define the continuous map $\Pi  \colon \prod_{j=1}^{\infty} I^{(j)}\to X$ by
$$\{\Pi(i_{1}, i_{2}, \dots) \} = \bigcap_{j=1}^{\infty} f^{(1)}_{i_{1}} \circ  f^{(2)}_{i_{2}} \circ \dots  \circ f^{(j)}_{i_{j}} (X),$$  
which is well-defined by the uniform contraction condition. 
We call $\Pi$ the coding map of $(\Phi^{(j)})_{j=1}^{\infty}$. 
Moreover, the image $J= \Pi(\prod_{j=1}^{\infty} I^{(j)})$ is called the limit set of $(\Phi^{(j)})_{j=1}^{\infty}$. 
\end{definition}

The concept of a non-autonomous IFS generalizes the classical (autonomous) IFS,  
where $\Phi^{(j)}$ remains the same for every $j \geq 1$.
Some of the non-autonomous IFSs considered in this paper are constructed by randomly selecting subsets $I^{(j)}$ from a fixed index set $I$ of an autonomous IFS $\{f_i\}_{i \in I}$. 
We refer to such non-autonomous systems as random IFSs.
As will be reviewed in Subsection~\ref{ssec:review}, 
much attention has been devoted to the dimension theory of these IFSs.

However, the topological properties of fractal sets have not been fully explored. 
With regard to the topology,
Hata \cite{Hata} established an equivalent condition for an autonomous limit set to be connected. 
Bandt and Mesing \cite{BM} studied the topology of IFSs ``of finite type'', which offers an interesting viewpoint. 
For autonomous IFSs of finite type, 
Luo and Xiong \cite{LX} gave an equivalent condition for the autonomous limit sets to be totally disconnected. 
For non-autonomous IFSs,  
Cristea \cite{Cristea} studied a concrete example of a randomly generated non-autonomous IFS,
focusing on when its limit set is connected and when it is totally disconnected.

Beyond connectedness, 
Sumi \cite{Sumi09} generalized Hata's result within a simplicial framework.  
He considered the nerve complex associated with the small copies appearing at the $k$th generation.  
This enabled a systematic investigation of the (co)homology groups of autonomous limit sets,
referred to as interaction (co)homology groups.
In addition to topological properties,
he also defined a new notion of dynamical complexity for autonomous IFSs through topological invariants.

One of the main goals of the present paper is to extend Sumi's homology theory to the non-autonomous setting.
This extension is a natural generalization of Sumi’s framework,
and is expected to provide a unified tool for analyzing a wide class of systems,
including randomly generated fractals and other dynamical systems with stochastic or non-stationary behavior.

Another main purpose is to study the topological properties of the so-called fractal squares and their generalizations. 
A typical example is the following.

\begin{example}\label{ex:fracSqSimple}
Let $X = [0, 1]^{2}$, the unit square, and set $I = \{0, 1, 2\} \times \{0, 1, 2\}$. 
For each $\mathbf{i} = (i_{1}, i_{2}) \in I$, 
define a contractive map $f_{\mathbf{i}}  \colon X \to X$ by 
$$f_{\mathbf{i}}(x, y) = \left(\frac{x +i_{1}}{3}, \frac{y +i_{2}}{ 3}\right).$$
For every $j \geq 1$, let $I^{(j)} \subset I$ be a non-empty subset. 
Then $\Phi^{(j)} = \{f_{\mathbf{i}}\}_{\mathbf{i}\in I^{(j)}}$ forms a non-autonomous IFS 
$(\Phi^{(j)})_{j \geq 1}$. 
\end{example}

We call the limit set of Example~\ref{ex:fracSqSimple} a non-autonomous fractal square.
Studying non-autonomous fractal squares is challenging because they do not satisfy a structural condition (see Definition~\ref{def:pub}) introduced in \cite{Sumi09}. 

\begin{figure}[h]
    \centering
    \includegraphics[width=0.3\linewidth]{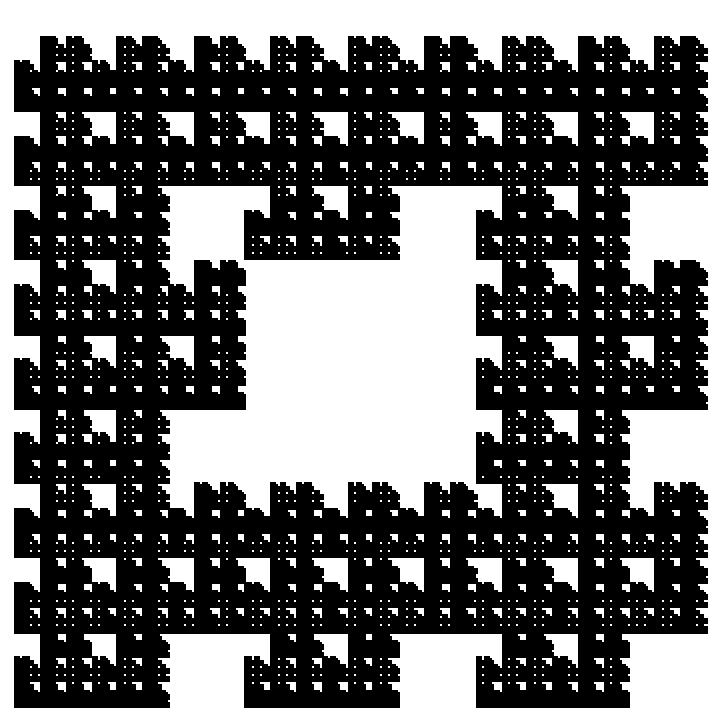}
    \includegraphics[width=0.3\linewidth]{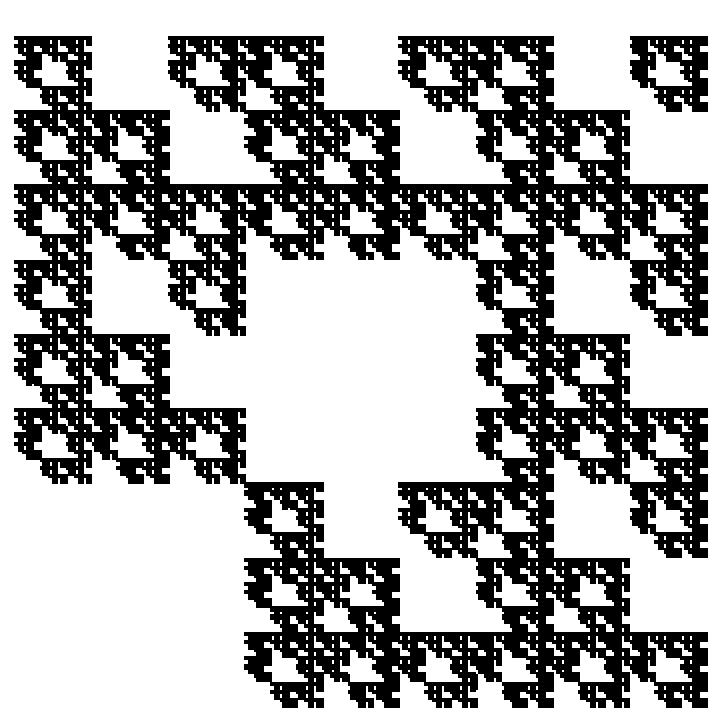}
    \includegraphics[width=0.3\linewidth]{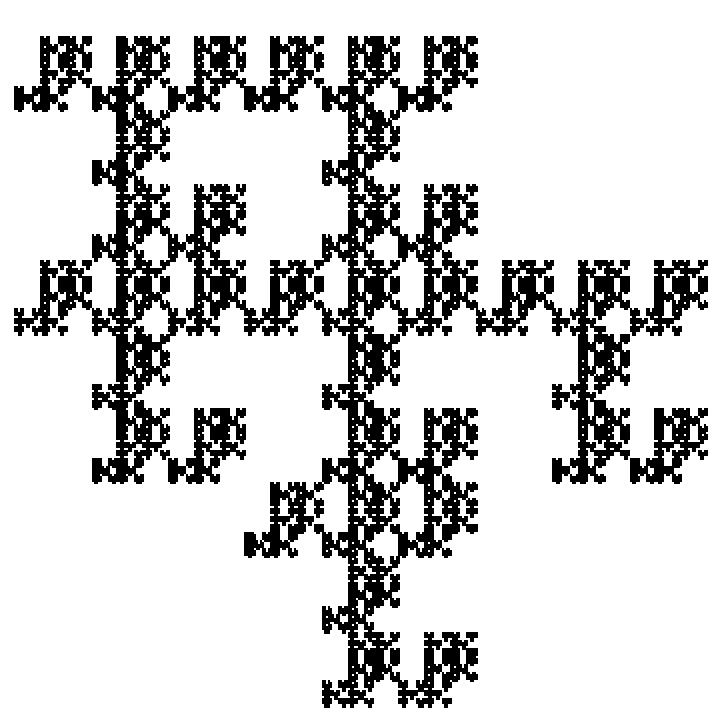}
    \caption{Samples of non-autonomous fractal squares.
    These are randomly constructed as in Theorem \ref{th:randomSqDim2}, 
    with parameters $r = 1$, $2$, $3$, respectively, from left to right, and with  $n_1=n_2=3$.}
    \label{fig:fracSq}
\end{figure}

Autonomous fractal squares have been extensively studied in the literature. 
Their topological properties have been widely investigated, including
(total dis)connectedness \cite{Ro}, Lipschitz equivalence \cite{LL, RW}, and H\"older equivalence \cite{ZL}. 
In a notable contribution, Xiao \cite{Xiao} investigated conditions under which an autonomous fractal square has finitely many connected components.
While Xiao's approach was partly a reworking of Sumi's theory, 
he introduced new ideas that enabled a precise count of the number of components. 
See also the survey \cite{LR} by Luo and Rao.

In this paper, we develop a topological theory for random IFSs and apply it to the study of non-autonomous fractal squares.
We establish conditions under which the limit set is (totally dis)connected.
Our result provides an answer to an analog of Mandelbrot’s percolation problem (to be reviewed in the next subsection).
We also compute the first \v{C}ech homology group $\check{H}_1$ of non-autonomous fractal squares. 

The passage from autonomous to non-autonomous systems is not merely formal. In the autonomous setting, the same pattern of overlaps is repeated at every scale. By contrast, in a non-autonomous IFS, both the pieces appearing at a given stage and the manner in which they intersect may change from one stage to the next. Consequently, the topology of the limit set can no longer be described by repeatedly applying a single finite combinatorial model. This difficulty is particularly pronounced for fractal squares, whose pieces may intersect along sets of positive dimension and which lie beyond the scope of several earlier methods.

To overcome this difficulty, we construct a sequence of simplicial models that record the overlap structure at each scale. We show that these models retain enough information to recover the \v{C}ech homology and cohomology of the limiting set. For concrete calculations, we further decompose each model so as to distinguish the homology inherited from individual pieces from the new homology created by intersections between different pieces. The resulting exact sequences provide recursive formulas for the homology of the finite-stage approximations.

Probability theory also allows us to move beyond the analysis of individual realizations and determine the behavior shared by almost every sequence. A central insight is to regard the successive stages at which new homological cycles appear as renewal times. This viewpoint converts the apparently irregular evolution of topological complexity into a tractable asymptotic problem. By combining renewal theory with the recursive homological formulas, we determine the precise almost-sure exponential growth rate of the Betti numbers.

\subsection{Main theorems}
\setcounter{res}{0}
In Section \ref{sec:prelim}, we present a probabilistic result on totally disconnected limit sets generated by random IFSs.

\begin{result}[Theorem \ref{th:A}]\label{res:prelim}
    Let $\{f_i\}_{i \in I}$ be an autonomous IFS which is post-critically countable (see Definition \ref{def:pcf}).
    Suppose that $f_i$ is injective for every $i\in I.$
    We choose the index sets $I^{(j)} \subset I$ independently and according to a fixed distribution such that the probability of $i \notin I^{(1)}$ is positive for every $i \in I$.   
	Then almost surely, 
	the resulting limit set of $(\{f_i\}_{i \in I^{(j)}})_{j=1}^{\infty}$ is totally disconnected. 
\end{result}

Theorem~\ref{res:prelim} shows that, in the post-critically countable regime, random
deletion typically destroys all symbolic identifications and forces the
limit set to be totally disconnected. Fractal squares lie beyond this
regime: their overlaps are governed by infinitely many post-critical
configurations. Their analysis therefore requires a different mechanism,
capable of retaining and organizing the overlap structure rather than
eliminating it.

In Section \ref{sec:sCpx}, 
we establish a 
simplicial
framework in which the limit sets of non-autonomous IFSs can be studied.
Using a certain kind of self-similarity (Lemma \ref{lem:Huchinson}), 
we construct a nested sequence of simplicial complexes $\mathcal{N}_{j, k}$ (Definition \ref{def:CechSumi}) 
that capture the overlapping structure of the IFS.
The main result here identifies the \v{C}ech homology group with the limit of the homology groups of these complexes.

\begin{result}[Theorem \ref{th:cechsumi} and Remark \ref{rem:CohomCoeff}]\label{res:CechSumi}
    Let $G$ be an  abelian group and consider homology groups with coefficients in $G$. Then,  
    the group $\varprojlim_{k}{H}_q(\mathcal{N}_{0, k}; G)$ defined as the inverse limit of homology groups is isomorphic to 
    the \v{C}ech homology group $\check{H}_q(J; G)$ of the limit set $J$ 
    for every $q \geq 0$. 
    Dually, $\varinjlim_{k}{H}^{q}(\mathcal{N}_{0, k}; G)$ is isomorphic to the \v{C}ech cohomology group $\check{H}^{q}(J; G).$
\end{result}

We call $\varprojlim_{k}{H}_{\ast}(\mathcal{N}_{0, k})$ and $\varinjlim_{k}{H}^{\ast}(\mathcal{N}_{0, k})$ 
the \v{C}ech-Sumi homology and cohomology groups for the non-autonomous IFS  $(\Phi^{(j)})_{j=1}^{\infty}$, respectively. 
The theorem above generalizes Sumi's result to the non-autonomous setting; see \cite[Remark~2.42]{Sumi09}. 

In Section \ref{sec:conn}, 
we establish Theorem~\ref{th:ConToInverselimit}, which shows correspondence between components of the limit set and components of simplicial complexes. 
As a consequence, we  obtain an intrinsic sufficient condition for total disconnectedness and a Hata-type result. 

\begin{result}[Corollary \ref{cor:totDisc}]\label{res:totDisc}
	If 
    $$\lim_{k \to \infty} c^{k} \max\{\#V(\mathcal{K}) \colon \mathcal{K} \text{ is a component of } \mathcal{N}_{0, k}\} = 0,$$ 
    then the limit set $J$ is totally disconnected. 
    Here,         
    $c$ is the uniform upper bound of the Lipschitz constants in Definition~\ref{def:NIFS}, 
    $V(\mathcal{K})$ denotes the vertex set of a simplicial complex $\mathcal{K}$,  and
    $\# A$ denotes the number of elements of a set $A$. 
\end{result}

\begin{result}[Corollary \ref{cor:Hata}]\label{res:Hata}
    The following are equivalent.
	\begin{enumerate}
		\item The limit set $J$ is connected. 
		\item For every $k > 0$, 
        the simplicial complex $\mathcal{N}_{0, k}$ is connected. 
	\end{enumerate} 
\end{result}

\begin{rem}\label{rem:ConnectedButNotPathConnected}
    For an autonomous IFS, the limit set is locally connected and path-connected if it is connected \cite{Hata}. However, Proposition~\ref{prop:TopSineCurve} shows that a non-autonomous limit set can be neither path-connected nor locally connected even if it is connected. 
\end{rem}

To this end, we give a sufficient condition for local connectedness and path-connectedness in Section~\ref{sec:exact}. 

\begin{result}[Proposition~\ref{prop:localConn} and Corollaries~\ref{cor:SuffConn} and \ref{cor:CompIsPathCon}]\label{res:PathConn}
     If $\mathcal{N}_{k, k+1}$ is connected for every $k > 0$, 
     then $J$ is locally connected. 
     Moreover, every connected component of $J$ is path-connected.
     If, in addition, $\mathcal{N}_{0,1}$ is connected, then $J$ is connected and path-connected.
\end{result}

In Section \ref{sec:exact}, 
we also define a suitable subcomplex $\mathcal{M}_{j, k, \ell}$ in Definition \ref{def:subcomplex}. 
The following exact sequence allows us to compute homology groups inductively.

\begin{result}[Theorem \ref{th:exact}]\label{res:exact}
	For the subcomplex $\mathcal{M}_{j, k, \ell}$ of $\mathcal{N}_{j, \ell}$, 
    there is a long exact sequence of homology groups
	\begin{equation*}
	\begin{tikzcd}[column sep=1em]
	\cdots \arrow[r, "\partial"] & {H}_q(\mathcal{M}_{j, k, \ell}) \arrow[r] & {H}_q(\mathcal{N}_{j, \ell}) 	\arrow[r] & {H}_q(\mathcal{N}_{j, \ell}, \mathcal{M}_{j, k, \ell})
	\arrow[r, "\partial"] & {H}_{q-1}(\mathcal{M}_{j, k, \ell}) \arrow[r] & \cdots.
	\end{tikzcd}
	\end{equation*}
\end{result}

Finally, Section \ref{sec:ex} applies this theory to non-autonomous fractal squares and their higher-dimensional analogs.
Definition~\ref{ex:fracSq} introduces a class of $d$-dimensional fractal cubes constructed by subdividing the $x_k$-coordinate direction into $n_k$ equal parts for each $k=1,\ldots,d$. 
More precisely, at each iteration, we randomly remove $r$ of the resulting small cubes and retain the maps corresponding to the remaining cubes.
Throughout Section~\ref{sec:ex}, we study non-autonomous IFSs arising from this construction. 

The following theorem summarizes the connectedness and homology of the resulting limit set $J$. 

\begin{result}[Theorems \ref{th:randomSq}, \ref{th:H0Growth}, and \ref{th:randomSqDim2}, together with Remark \ref{rem:tosionFree}]\label{res:ex}
    Consider a non-autonomous fractal $d$-cube. 
    Let $r \in \NN$ such that $1 \leq r \leq (\prod_{k=1}^{d}n_{k}) -1$. 
    Suppose that the sets $I^{(j)}$ are chosen independently according to the uniform distribution 
    so  that $\# (I\setminus I^{(j)}) = r$. 
    Then the limit set $J$ of $(\Phi^{(j)})_{j=1}^{\infty}$ satisfies the following. 
    \begin{enumerate}[label=($d$.\arabic*)]
        \item If $r < d$, then $J$ is connected and locally connected. 
        \item If there exists $k =1, 2, \dots, d$ such that $r < \prod_{\ell \neq k} n_{\ell}$, 
        then $J$ contains a line segment which connects the face  $x_k = 0$ to the opposite face $x_k = 1$. 
        \item If there exists $k =1, 2, \dots, d$ such that $r \geq \prod_{\ell \neq k} n_{\ell}$, then almost surely every connected component $C$ of $J$ has its projection $\pi_{k}(C)$ onto the $x_k$-coordinate equal to a single point. 
        \item If $r \geq \prod_{\ell \neq k} n_{\ell}$ for every $k =1, 2, \dots, d$, then almost surely $J$ is totally disconnected. 
    \end{enumerate}

    Moreover, 
    in the planar case $d = 2$ we obtain detailed homology results:  
    The limit set $J$ satisfies $\check{H}_q(J) = 0$ for every $q \geq 2$. 
    Furthermore, we have the following. 
    \begin{enumerate}[label=($2$.\arabic*)]
        \item If $r = 1,$ then $\check{H}_0(J) \cong \ZZ$. In addition to this, we have the following. 
        \begin{enumerate}
            \item If $n_1 = n_2 = 2,$ then $\check{H}_1(J) = 0$ almost surely.
            \item If $(n_1, n_2) \neq (2, 2),$ then 
            $$\lim_{k \to \infty}\frac{1}{k}\log (\mathrm{rank} H_1(\mathcal{N}_{0, k})) = \log(n_1n_2 - r)$$
            almost surely. 
        \end{enumerate}
        \item If $2 \leq r < \min\{n_{1}, n_{2}\}$,   
        then 
        $$\lim_{k \to \infty}\frac{1}{k}\log (\mathrm{rank} H_0(\mathcal{N}_{0, k})) = \log(n_1n_2-r)$$ 
        and 
        $$\lim_{k \to \infty}\frac{1}{k}\log (\mathrm{rank} H_1(\mathcal{N}_{0, k})) = \log(n_1n_2-r)$$
        almost surely. 
        \item If $n_{1} \leq r < n_{2}$ (resp. $n_{2} \leq r < n_{1}$), then almost surely, every connected component of $J$ is a horizontal (resp. vertical) line segment. One component has length $1$, and the others may degenerate to points.  
        \item If $r \geq \max\{n_{1}, n_{2}\}$, then almost surely $J$ is totally disconnected. 
    \end{enumerate}
    Furthermore, for all $k > 0$ and $q \geq 0$, each homology group
    ${H}_q(\mathcal{N}_{0, k})$ is free, and hence isomorphic to the cohomology group ${H}^q(\mathcal{N}_{0, k})$. 
\end{result}

\begin{rem}
    By the Alexander duality theorem \cite[Theorem 6.2.16]{Spa}, 
    the topology of a set $J$ determines that of its complement.
    More precisely, 
    the \v{C}ech cohomology group $\check{H}^{d-q-1}(J; G)$ is isomorphic to the reduced homology group $\tilde{H}_{q}(\RR^d\setminus J;G)$ of the complement. 
    For a non-autonomous fractal rectangle $J \subset \RR^2$, 
    setting \(q=0\), we deduce from Theorem \ref{res:ex} that the number of bounded connected components of \(\mathbb R^2\setminus J\) is either zero or infinite,  
    according as $\mathrm{rank} \check{H}_1(J)$ is zero or infinite.
\end{rem}

In the planar case, our results give a complete almost-sure phase
diagram. Depending on the number of deleted rectangles, the limit set
is connected, has infinitely many connected components, decomposes into
parallel line segments, or is totally disconnected. Thus the model
exhibits several distinct topological regimes, separated by explicit
and sharp thresholds, whereas the corresponding picture in higher dimensions remains open.

Theorem \ref{res:ex} concerns not only topological properties but also dynamical aspects.
Namely, the growth rate of the rank of the (co)homology groups quantifies the dynamical complexity of IFSs.
This invariant provides new insight into how topological complexity reflects dynamical behavior, making it a particularly meaningful quantity in the study of IFSs.
See also \cite[Theorem 3.36]{Sumi09}. 

\subsection{Related work from various fields}\label{ssec:review}
To clarify our contribution, we review related results from several areas.

\subsubsection{Fractal dimensions}
In fractal geometry, much attention has been devoted to the dimension theory of IFSs.
For the non-autonomous setting, Rempe-Gillen and Urba\'nski \cite{RU16} showed that 
the Hausdorff dimension of the limit set is given by Bowen’s formula under the separating condition.
Further results 
for dimension theory of non-autonomous IFSs
can be found in \cite{FT, GM, HZ, KR}. 

The relationship between dimension and topology is not very strict. 
It is known that if the Hausdorff dimension is less than one, then the set is totally disconnected.
The converse, however, does not hold: one can construct a space with arbitrarily prescribed dimension which is homeomorphic to the Cantor set \cite{Ishiki}.
Another general fact is that topological dimension provides a lower bound for the Hausdorff dimension, although this estimate is usually far from sharp.

In dimension theory, imposing a separation condition often simplifies the analysis. In contrast, our work demonstrates that 
rich topological structures come from
the overlaps of small pieces. 
See also the first author's work \cite{Nak} for non-autonomous IFSs without the separating condition. 
Moreover, our results hold even if we do not assume conformality.

While the dimension theory of non-autonomous IFSs has been studied, 
the topological aspect of non-autonomous IFSs remains almost completely unexplored.
The examples developed in this paper, in which the homological growth rate agrees with entropy, point toward a potentially deeper connection with dimension theory that remains to be explored.

\subsubsection{Mandelbrot fractal percolation}\label{sssec:perc}
While we are partly interested in randomly generated non-autonomous fractal squares, 
another type of random fractal has also been studied by many authors. 
A central problem concerns the connectivity properties of the following random fractal. 

\begin{example}[Fractal percolation]\label{ex:percolation}
    Let $X$ be the unit square. 
    Divide the square into $3\times3$ equal subsquares. 
    Independently for each subsquare, retain it with probability $p$ and discard it with probability $1-p$.
    The same procedure is then applied recursively and independently to each retained subsquare.
    The intersection of the sets retained at all stages is called a Mandelbrot percolation fractal, or a fractal percolation set.
    This construction differs from that in Example~\ref{ex:fracSqSimple}
\end{example}

\begin{figure}[h]
    \centering
    \includegraphics[width=0.3\linewidth]{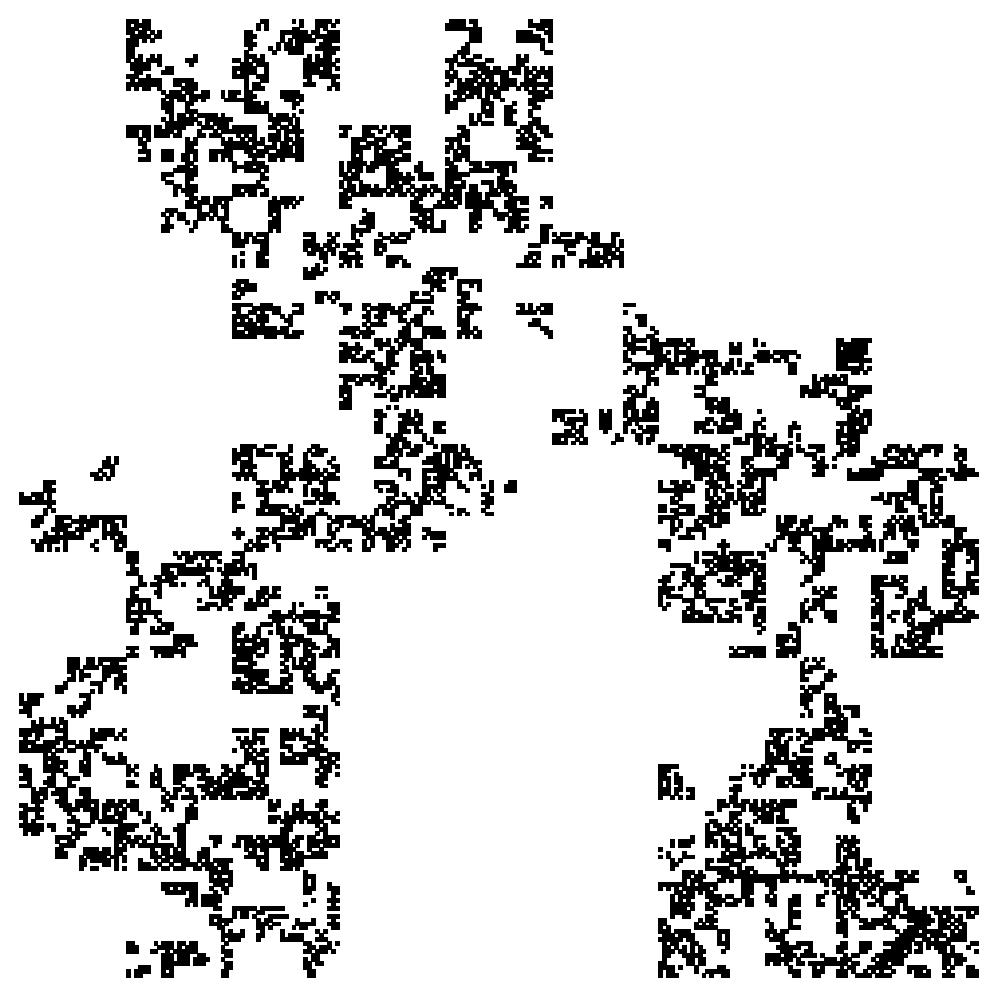}
    \includegraphics[width=0.3 \linewidth]{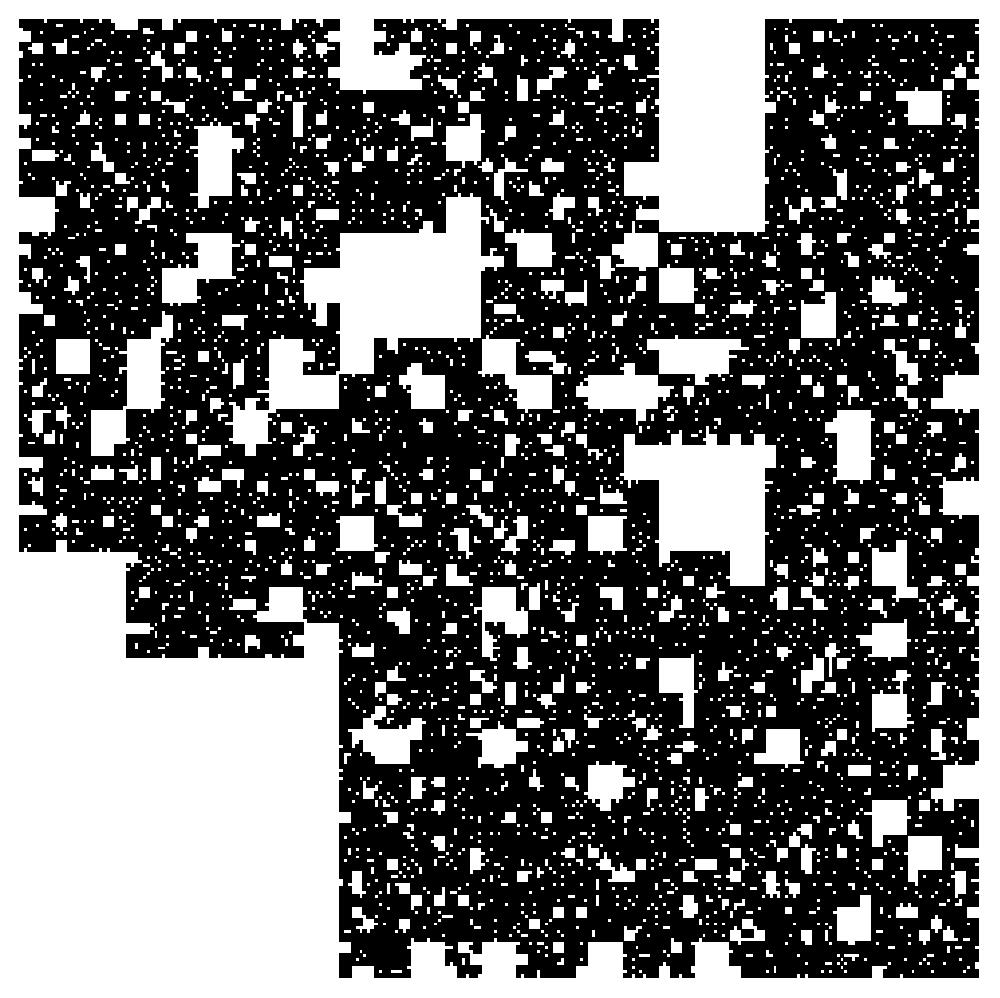}
    \caption{Samples of Mandelbrot percolation fractals with $p = 6/9$ and $p = 8/9$.}
    \label{fig:percolation}
\end{figure}

Since the decision to retain or discard each subsquare is made independently at each subsquare, 
the resulting set may exhibit a more intricate structure than the limit set of Example \ref{ex:fracSqSimple}.
Compare Figures \ref{fig:fracSq} and \ref{fig:percolation}.

A substantial body of work has been devoted to the critical phase transition in fractal percolation, particularly in connection with the existence of a path joining two opposite sides of the unit square. For further developments stemming from the pioneering work of Chayes et al.~\cite{CCD}, we refer the reader to Falconer~\cite[Section~15.2]{FalFG} and the references therein.

Falconer also studied the Mandelbrot percolation fractal and
he observed that, by selecting the retained squares appropriately, 
one can construct finite-stage approximations whose homological complexity 
increases strictly from one stage to the next, so that the limiting set has infinite connectivity; see  \cite[Example 11.5]{F}.
Although this concerns a specially arranged construction rather than the 
random non-autonomous model considered here, it anticipates the emergence 
of increasingly rich homological structure across scales. 

Theorem~\ref{res:ex} provides a systematic and quantitative counterpart to this 
observation: for natural random non-autonomous systems, we determine both 
the homology of the limiting set and the exponential growth rate of the 
homology groups of its finite-stage approximations.
This quantitative result is obtained by combining a simplicial framework adapted to IFSs with general tools from homology theory.

\subsubsection{Dynamical systems} 
Non-autonomous IFSs also arise naturally in dynamics, where they model
inverse branches along non-stationary or random orbits and provide a
flexible framework for studying dynamically defined repellers.
Indeed, the concept of a non-autonomous IFS was developed to obtain estimates of the Hausdorff dimension of the Julia set for transcendental functions \cite{RU16}.

The second author introduced the notion of ``stochastic bifurcation'' for random iterations of quadratic polynomial maps,
and investigated the (total dis)connectedness of the random Julia set \cite{W24}.
These works motivate our study of non-autonomous IFSs, with the goal of analyzing the topological properties of their limit sets.

\subsubsection{Topology of wild spaces}
From the viewpoint of pure topology, the present work can be regarded as a contribution to the algebraic topology of wild spaces.
Barratt and Milnor pointed out that singular homology may behave anomalously \cite{BMil}.
Related studies on wild spaces include 
investigations of the first singular homology group of the Hawaiian earring \cite{EK}, 
as well as the fundamental group of the Sierpi\'nski gasket \cite{ADTW}.

It is also worth noting that Sumi's homology theory shares a similar philosophy with the construction of the Anderson-Putnam complex for tiling spaces \cite{Sad}.
For self-similar tilings themselves, many topological aspects have been studied; 
see the survey \cite{AT}, for example.

While some of these studies emphasize the pathological aspects of wild spaces,
the present work seeks to establish a coherent theory for fractals.
Fractals are indeed wild spaces,
but their self-similarity endows them with a structural order.
This combination of wildness and order makes them especially valuable objects of study within the field of topology. 

The principal contribution of the present work is the construction of an abstract simplicial and homological framework for the limit sets of non-autonomous IFSs. This framework provides a common foundation for the systematic analysis of diverse concrete examples and, we hope, will facilitate further investigations into the topology of fractal spaces.

\subsection{Organization of the paper}
Section \ref{sec:prelim} proves the probabilistic result on random IFSs (Theorem \ref{res:prelim}). 
While this result places the developments of the subsequent sections in a broader probabilistic perspective, the section is logically independent of the remainder of the paper.
Section \ref{sec:sCpx}  introduces the \v{C}ech-Sumi (co)homology groups for non-autonomous IFSs and proves Theorem \ref{res:CechSumi}. 
In Section~\ref{sec:conn}, we develop a general theory on connectedness using simplicial methods (Theorems~\ref{res:totDisc} and \ref{res:Hata}).
In Section~\ref{sec:exact}, we construct the subcomplex and present Theorems\ref{res:PathConn} and \ref{res:exact}. 
Section \ref{sec:ex} is devoted to an application: the computation of the homology groups of fractal squares.

\section{Preliminary remarks on random IFSs}\label{sec:prelim}
An interesting class of non-autonomous IFSs is one in which index sets $I^{(j)}$ are chosen from a fixed index set $I$.
In this section, we show that the randomly generated limit set $J$ of such a non-autonomous IFS is totally disconnected if the total IFS is post-critically finite. 
Let us first define a post-critically finite IFS; 
see \cite[Subsection 1.3]{Ki} for the details.

\begin{definition}\label{def:pcf}
	Let $\Phi = \{f_i\}_{i \in I}$ be an autonomous IFS, that is a non-autonomous IFS with $\Phi^{(j)} = \Phi$ for every $j \geq 1$. 
	We define the critical set by $C = \bigcup_{i \neq i' \in I} \left( f_i(J) \cap f_{i'}(J) \right)$. 
	Consider the backward image $\Pi^{-1}(C)$ under the coding map $\Pi  \colon \prod_{j=1}^{\infty} I\to J$, 
	and consider the left-shift $\sigma  \colon \prod_{j=1}^{\infty} I \to \prod_{j=1}^{\infty} I$. 
	We define the post-critical set by $\bigcup_{n=1}^{\infty} \sigma^{n}(\Pi^{-1}(C))$. 
	We say that $\Phi$ is post-critically finite if the post-critical set is finite, 
	and that $\Phi$ is post-critically countable if the post-critical set is at most countable. 
\end{definition}

For instance, 
the Sierpi\'nski gasket, the Koch curve, and the pentakun \cite[Example 3.28]{Sumi09} are post-critically finite, 
but the Sierpi\'nski carpet is not post-critically countable.

Sumi introduced the class of postunbranched IFSs \cite[Definition 3.22]{Sumi09}
to derive the  recursive formula for the rank of the cohomology groups.

\begin{definition}\label{def:pub}
	Let $\Phi = \{f_i\}_{i \in I}$ be an autonomous IFS. 
	Set $C_{i, i'} = f_i(J) \cap f_{i'}(J)$. 
	We say that $\Phi$ is postunbranched 
	if for any $(i,i') \in I^{2}$ such that  $i \neq i'$ and $C_{i, i'} \neq \emptyset$, 
	there exists a unique $\mathfrak{i} \in  \prod_{j=1}^{\infty} I$ such that $f_i^{-1}(C_{i, i'})  = \{\Pi(\mathfrak{i})\}$. 
\end{definition}

There are several examples of IFSs that are post-critically finite but not postunbranched, 
see \cite[Example 2]{BM}. 
However, the following holds. 

\begin{lemma}\label{lem:pub}
	Let $\Phi$ be an autonomous IFS.  
	If it is postunbranched, then it is post-critically countable.
\end{lemma}

\begin{proof}
	Let $\Phi = \{f_i\}_{i \in I}$ be a postunbranched IFS. 
	With the notation of definitions above, 
	the critical set is  $C = \bigcup_{i \neq i' \in I} C_{i, i'}$. 
    For each $i$ and $i'$ with $i \neq i'$, let $\mathfrak{u}_{i, i'}$ satisfy $f_i^{-1}(C_{i, i'})=\Pi(\mathfrak{u}_{i, i'}).$ We shall show $\sigma(\Pi^{-1}(C_{i, i'}))\subset \{\mathfrak{u}_{v, w'}\}_{v\neq w'}$ for every $i$ and $i'$.
    
   To prove this, take $\mathfrak{v}\in \sigma(\Pi^{-1}(C_{i, i'})).$ Then for some $v \in I$ we have $v \mathfrak{v}\in \Pi^{-1}(C_{i, i'}).$ Since $\Pi(v \mathfrak{v})\in C_{i, i'}\cap f_{v}(J),$ there exists $w \in \{i, i'\}$ such that $w\neq v$ and $\Pi(v \mathfrak{v})=f_v(\Pi(\mathfrak{v}))\in C_{v, w}.$ Therefore, we have $\Pi( \mathfrak{v})=f_{v}^{-1}(C_{v, w})=\Pi(\mathfrak{u}_{v, w}),$ which implies $\mathfrak{v}=\mathfrak{u}_{v, w}$ by the uniqueness ensured by the postunbranched IFS.  
   Hence, $\bigcup_{n=1}^{\infty} \sigma^{n}(\Pi^{-1}(C_{i, i'}))=\bigcup_{n=0}^{\infty}\sigma^n (\{\mathfrak{u}_{i, i'}\}_{i\neq i'})$ is a countable set for every $i$ and $i'$. 
\end{proof}

\begin{theorem}\label{th:A}
	Let $\{f_i\}_{i \in I}$ be a post-critically countable IFS such that $f_i$ is injective for every $i\in I.$
    Let $\mu$ be a probability measure on the set of all non-empty subsets of $I$, 
	and let $\left(I^{(j)}\right)_{j=1}^\infty$ be an i.i.d.\ sequence with common law $\mu$.
	Suppose that for every $i \in I$, we have 
    $\mu\left(\{I' \subset I \colon i \notin I'\}\right) > 0$. 
	Then, for $\mu^{\otimes\infty}$-almost every sequence,
	the corresponding non-autonomous limit set of $(\{f_i\}_{i \in I^{(j)}})_{j=1}^{\infty}$ is totally disconnected. 
\end{theorem}

\begin{proof}
	Consider the coding map $\Pi  \colon \prod_{j=1}^{\infty} I\to X$ of the autonomous IFS $\{f_i\}_{i \in I}$.
	Let $\mathcal{P}$ be the post-critical set of the autonomous IFS $\{f_i\}_{i \in I}$, which is at most countable by the assumption.
	 Fix $m\in \mathbb N$ and $\mathfrak{i} = (i_m, i_{m+1}, \dots) \in \mathcal{P}$. Then, the probability that $\mathfrak{i} \in \prod_{j=1}^{\infty} I^{(m+j-1)}$ is zero 
	since 
     $\mu\left(\{J \subset I \colon i_{m+j-1} \in J\}\right) < 1$ for every $j \geq 1$. 
	Thus, the probability that $\mathcal{P} \cap \bigcup_{m=1}^{\infty}\prod_{j=1}^{\infty} I^{(m+j-1)} \neq \emptyset$ is zero.

	Suppose now that $\mathcal{P} \cap \bigcup_{m=1}^{\infty}\prod_{j=1}^{\infty} I^{(m+j-1)}= \emptyset$. 
	Since the coding map $\Pi  \colon \prod_{j=1}^{\infty} I^{(j)}\to X$ of the non-autonomous IFS is 
	the restriction of the coding map $\Pi  \colon \prod_{j=1}^{\infty} I\to X$ of the autonomous IFS $\{f_i\}_{i \in I}$, 
	it suffices to show that the restricted $\Pi  \colon \prod_{j=1}^{\infty} I^{(j)}\to X$ is injective.

	Suppose that $\mathfrak{i} \neq \mathfrak{i}' \in  \prod_{j=1}^{\infty} I^{(j)}$ satisfies $\Pi(\mathfrak{i}) = \Pi(\mathfrak{i}')$. 
	Let $\mathfrak{i} = (i_1, i_2, \dots)$ and $\mathfrak{i}' = (i'_1, i'_2, \dots)$. 
	Then there exists $n \geq 1$ such that $i_j=i'_j$ for every $j=1,..., n-1$ and $i_n \neq i'_n$.
	Consider the shifted symbols 
	$\sigma^{n-1}(\mathfrak{i}) = (i_{n}, i_{n+1}, \dots)$ and 
	$\sigma^{n-1}(\mathfrak{i}') = (i'_{n}, i'_{n+1}, \dots)$. 
	Since $\Pi(\mathfrak{i}) = \Pi(\mathfrak{i}')$, by the injectivity of the maps $f_i $ $(i\in I),$ we have $\Pi(\sigma^{n-1}(\mathfrak{i})) = \Pi(\sigma^{n-1}(\mathfrak{i}'))$.
	Then we have $f_{i_n}(\tilde{J}) \cap f_{i'_n}(\tilde{J}) \neq \emptyset$, where $\tilde{J}$ is the limit set of the autonomous IFS $\{f_i\}_{i \in I}$. Hence, $\sigma^n(\mathfrak{i})\in \mathcal P$. This contradicts the fact $\mathcal{P} \cap \bigcup_{m=1}^{\infty}\prod_{j=1}^{\infty} I^{(m+j-1)}= \emptyset$.
	Thus, the  restricted map $\Pi  \colon \prod_{j=1}^{\infty} I^{(j)}\to X$ is injective, 
	and hence the non-autonomous limit set $J$ is totally disconnected. 
\end{proof}

\begin{example}
    Let $p_0, p_1, p_2\in\RR^2$ be the vertices of  an equilateral triangle. 
    Define $g_i : \RR^2\to\RR^2$ by $g_i(x)=(x-p_i)/2 + p_i$ for every $i=0,1,2.$
    The limit set $X$ of the autonomous IFS $\{g_i\}_{i=0,1,2}$ is called the Sierpi\'nski gasket. Fix $n\in\NN$ and consider the $n$th iterates: 
    for every $\mathbf{i} = (i_1, \dots, i_n)\in\{0, 1,2\}^n$, define $f_\mathbf{i} = g_{i_1}\circ\dots\circ g_{i_n}$. 
    Then, the autonomous IFS $\{f_\mathbf{i}\}_{\mathbf{i}\in\{0,1,2\}^n}$ is post-critically finite. 

    Let $I:=\{0,1,2\}^n$ and let $\mathcal{P}_1(I)$ be the set of all subsets $I'$ of $I$ such that $\#(I\setminus I') = 1$. 
    Let $\mu$ be a  probability measure on $\mathcal{P}_1(I)$ which has the same mass on every element. 
    Then, for $\mu^{\otimes\infty}$-almost every $(I^{(j)})_{j=1}^\infty$,
	the limit set $J$ of $(\{f_\mathbf{i}\}_{\mathbf{i} \in I^{(j)}})_{j=1}^{\infty}$ is totally disconnected. 
    We also have that the Hausdorff dimension satisfies $\dim_\mathrm{H}X = \log3/\log2$ and $\dim_\mathrm{H}J = \log(3^n -1)/\log(2^n)$ by \cite[Theorem 5.3]{RU16}. 
\end{example}

In Section \ref{sec:ex}, 
we will consider non-autonomous fractal squares, whose total IFS is not post-critically countable.

\section{Definitions of the simplicial and homological framework}\label{sec:sCpx}
In this section, we define the sequence of simplicial complexes and develop the homological theory in a general form. 
Henceforth, we consider a non-autonomous IFS $(\Phi^{(j)})_{j=1}^{\infty}$  on a compact metric space $X$ 
whose limit set is denoted by $J$ as in Definition \ref{def:NIFS}.  

\begin{definition}
    Let $Y$ be a topological space. 
    For a collection $\mathfrak{U}$ of subsets of $Y$, we denote the nerve by $N(\mathfrak{U})$. 
    Namely, $N(\mathfrak{U})$ is the (abstract) simplicial complex 
    whose simplexes are finite non-empty subsets of $\mathfrak{U}$ with non-empty intersection.
\end{definition}

Recall that the $q$th \v{C}ech homology group is defined as the inverse limit
$\check{H}_q(Y) = \varprojlim_{\mathfrak{U}} H_{q} (N(\mathfrak{U}))$, where $H_{q}(\mathcal{K})$ is the $q$th homology group of a simplicial complex $\mathcal{K}$ and 
where $\mathfrak{U}$ runs over all finite open coverings of $Y$ ordered by refinement. 
See the book \cite{W70} by Wallace for the details. 
One can similarly define the $q$th \v{C}ech homology group $\check{H}_{q}(Y; G)$ with coefficients in any abelian group $G$, 
but we shall mainly restrict ourselves to the integral \v{C}ech homology in this paper.

We will define a new homology group in Definition \ref{def:CechSumi}. 
For this purpose, we need the following lemma, which is a generalization of Hutchinson's theorem. 

\begin{definition}
    For every $j \geq 0$, define a non-autonomous IFS $(\Phi^{(j+k)})_{k=1}^{\infty}$ and denote its limit set by $J_{j}$. 
\end{definition}

\begin{lemma}\label{lem:Huchinson}
    We have $J = J_{0}$, and for every $0 \leq j < k$ we have 
     $$J_{j} = \bigcup_{(i_{j+1}, \dots, i_{k}) \in I^{(j+1)} \times \dots \times  I^{(k)}} 
     f^{(j+1)}_{i_{j+1}} \circ  \dots  \circ f^{(k)}_{i_{k}}(J_{k}).$$
\end{lemma}

\begin{proof}
    It is trivial that $J = J_{0}$. 
    We show $J_{0} = \bigcup_{i\in I^{(1)}}  f^{(1)}_{i}(J_{1})$.
    Then, by a similar argument, one can show $J_{j} = \bigcup_{i\in I^{(j+1)}}  f^{(j+1)}_{i}(J_{j+1})$ for every $j \geq 1$, and hence the lemma follows. 
    
    Denote by $\Pi \colon \prod_{j=1}^{\infty} I^{(j)}\to X$ and $\Pi' \colon \prod_{j=1}^{\infty} I^{(j+1)}\to X$ 
    the coding maps associated with $(\Phi^{(j)})_{j=1}^{\infty}$ and $(\Phi^{(j+1)})_{j=1}^{\infty}$, respectively. 
    For every $x \in J_{0}$, there exists $(i_{1}, i_{2}, \dots) \in \prod_{j=1}^{\infty} I^{(j)}$ such that $\Pi(i_{1}, i_{2}, \dots) = x$. 
    Let $x' = \Pi'( i_{2}, \dots) \in J_{1}$. 
    By definition, we have $\{\Pi'(i_{2}, \dots) \} = \bigcap_{j=2}^{\infty} f^{(2)}_{i_{2}} \circ \dots  \circ f^{(j)}_{i_{j}} (X)$,  thus $x = f^{(1)}_{i_{1}} (x') \in  f^{(1)}_{i_{1}}(J_{1})$. 
    This shows $J_{0} \subset \bigcup_{i\in I^{(1)}}  f^{(1)}_{i}(J_{1})$.
    Also,  
    for every $x' \in J_{1}$, there exists $(i_{2}, \dots) \in \prod_{j=1}^{\infty} I^{(j+1)}$ such that $\Pi'(i_{2}, \dots) = x'$.
    For every  $i\in I^{(1)}$, by concatenating $i$ and $(i_{2}, \dots)$, we set  $x = \Pi(i, i_{2}, \dots) \in J_{0}$. 
    Then $x = f^{(1)}_{i} (x')$, which implies $J_{0} \supset \bigcup_{i\in I^{(1)}}  f^{(1)}_{i}(J_{1})$. 
    This completes the proof. 
\end{proof}

\begin{definition}\label{def:CechSumi}
    For every $0 \leq j < k$, 
    we define the simplicial complex $\mathcal{N}_{j, k}$ as the nerve of the covering 
    $\{f^{(j+1)}_{i_{j+1}} \circ  \dots  \circ f^{(k)}_{i_{k}}(J_{k})\}_{(i_{j+1}, \dots, i_{k}) \in I^{(j+1)} \times \dots \times  I^{(k)}}$ of $J_j$. 
    Namely, we regard each $v = (i_{j+1}, \dots,  i_{k}) \in I^{(j+1)} \times \dots \times I^{(k)}$ as a vertex of $\mathcal{N}_{j, k}$, 
    and the set $\{v_{0}, v_{1}, \dots, v_{q}\}$ of mutually distinct vertices is a simplex of $\mathcal{N}_{j, k}$ 
    if and only if $\bigcap_{p=0}^{q} f_{v_{p}}(J_{k}) \neq \emptyset$. 
    Here, we define a map by $f_{v} = f^{(j+1)}_{i_{j+1}} \circ  \dots  \circ f^{(k)}_{i_{k}}$ for every $v = (i_{j+1}, \dots,  i_{k})$. 
     
    We define the simplicial map $\phi_{j, k}  \colon \mathcal{N}_{j, k+1} \to \mathcal{N}_{j, k}$ 
    so that $$\phi_{j, k}(i_{j+1}, \dots,  i_{k}, i_{k+1}) = (i_{j+1}, \dots,  i_{k}).$$  
    This simplicial map $\phi = \phi_{j, k}$ depends on $j$ and $k$, but the subscript will be omitted when it is clear from the context.
    
    For every $q \geq 0$, 
    the simplicial map induces a homomorphism 
    $$\phi_{*}  \colon H_{q}(\mathcal{N}_{j, k+1}) \to H_{q}(\mathcal{N}_{j, k})$$
    on the homology groups (with $\ZZ$ coefficients). 
    Consider the inverse limit of the inverse system $\{\phi_{*}  \colon H_{q}(\mathcal{N}_{0, k+1}) \to H_{q}(\mathcal{N}_{0, k})\}_{k = 1}^{\infty}$. 
    We call $\varprojlim_{k}{H}_q(\mathcal{N}_{0, k})$ the $q$th \v{C}ech-Sumi homology group of the non-autonomous IFS $(\Phi^{(j)})_{j=1}^{\infty}$.
\end{definition}

\begin{example}\label{ex:twoGenerator}
	Let $X = [ 0, 1 ]$ be the unit interval. 
	Let $I^{(j)} = \{a, b\}$ for every $j \geq 1$, where $a$ and $b$ are two distinct symbols.
	Define a map $f_{a}^{(j)}  \colon X \to X$ by $f_{a}^{(j)}(x) = 5x/7$ if $j$ is odd, and $f_{a}^{(j)}(x) = 2x/5$ if $j$ is even. 
	Also, define a map $f_{b}^{(j)}  \colon X \to X$ by $f_{b}^{(j)}(x) = 5x/7 + 2/7$ if $j$ is odd, and $f_{b}^{(j)}(x) = 2x/5 + 3/5$ if $j$ is even. 

	Then, for the non-autonomous IFS $(\{f_{a}^{(j)}, f_{b}^{(j)} \})_{j=1}^{\infty}$, 
	the limit sets are $J_j = [ 0, 1 ]$ if $j$ is even, 
	and $J_j = [ 0, 2/5 ] \cup [3/5, 1]$ if $j$ is odd.  
	In addition, the simplicial complexes are 
	\begin{align*}
		\mathcal{N}_{0, 1} = 
		\{
			&\{a\}, \{b\}, 
			\{a, b\}
		\}, \\
		\mathcal{N}_{0, 2} = 
		\{
			&\{(a, a)\}, \{(a, b)\}, \{(b, a)\}, \{(b, b)\}, \\
			&\{(a, a), (b, a)\}, \{(b, a), (a, b)\}, \{(a, b), (b, b)\}  
		\},\\
		\mathcal{N}_{1, 2} = 
		\{
			&\{a\}, \{b\}
		\}, \\
		\mathcal{N}_{1, 3} = 
		\{
			&\{(a, a)\}, \{(a, b)\}, \{(b, a)\}, \{(b, b)\}, \\
			&\{(a, a), (a, b)\}, \{(b, a), (b, b)\} 
		\},
	\end{align*} 
	and so on. 
	See the figures below.  

	\vspace{0pt}
	\begin{tikzpicture}
	\draw (0,0)--(1,0);
	\node (O) at (0,0) [left=12pt] {$\mathcal{N}_{0, 1}$};
	\coordinate (1) at (0,0) node at (1) [below] {$a$\phantom{b}};
	\coordinate (2) at (1,0) node at (2) [below] {$b$};
	\fill (1) circle (2pt) (2) circle (2pt);
	\end{tikzpicture}

	\vspace{10pt}
	\begin{tikzpicture}
	\draw (0,0)--(2,-1)--(4,0)--(6,-1);
	\node (O) at (0,0) [left= 24pt, below=12pt] {$\mathcal{N}_{0, 2}$};
	\coordinate (0) at (0,0) node at (0) [above] {$(a, a)$};
	\coordinate (1) at (2,-1) node at (1) [below] {$(b, a)$};
	\coordinate (2) at (4,0) node at (2) [above] {$(a, b)$};
	\coordinate (3) at (6,-1) node at (3) [below] {$(b, b)$};
	\fill (0) circle (2pt) (1) circle (2pt) (2) circle (2pt) (3) circle (2pt);
	\end{tikzpicture}

	\vspace{0pt}
	\begin{tikzpicture}
	\node (O) at (0,0) [left=12pt] {$\mathcal{N}_{1, 2}$};
	\coordinate (1) at (0,0) node at (1) [below] {$a\phantom{b}$};
	\coordinate (2) at (1,0) node at (2) [below] {$b$};
	\fill (1) circle (2pt) (2) circle (2pt);
	\end{tikzpicture}

	\vspace{0pt}
	\begin{tikzpicture}
	\draw (0,0)--(2,0) (4,0)--(6,0);
	\node (O) at (0,0) [left=24pt] {$\mathcal{N}_{1, 3}$};
	\coordinate (0) at (0,0) node at (0) [below] {$(a, a)$};
	\coordinate (1) at (2,0) node at (1) [below] {$(a, b)$};
	\coordinate (2) at (4,0) node at (2) [below] {$(b, a)$};
	\coordinate (3) at (6,0) node at (3) [below] {$(b, b)$};
	\fill (0) circle (2pt) (1) circle (2pt) (2) circle (2pt) (3) circle (2pt);
	\end{tikzpicture}
\end{example}

We now show one of our main results (Theorem \ref{res:CechSumi}). 
This is a generalization of \cite[Remark 2.42]{Sumi09}. 

\begin{theorem}
\label{th:cechsumi}
    There is an isomorphism between 
    the \v{C}ech-Sumi homology group $\varprojlim_{k}{H}_q(\mathcal{N}_{0, k})$ and 
    the \v{C}ech homology group $\check{H}_q(J)$ for every $q \geq 0$. 
\end{theorem}

\begin{proof}
	For every $k \geq 1$, 
	denote by $V(\mathcal{N}_{0, k})$ the set of all vertices of a simplicial complex $\mathcal{N}_{0, k}$.  
	Then we have 
	$$J = \bigcup_{v \in V(\mathcal{N}_{0, k})} f_v(J_{k})$$
	by Lemma \ref{lem:Huchinson}.
	For every set $f_v(J_{k})$, denote its open $\delta$-neighborhood by
	$N_{\delta}(f_v(J_{k}))$. 
	
	Since $X$ is compact, 
	there exists a small $\delta_k > 0$ such that the nerve of $\mathfrak{U}_k = \{J \cap  N_{\delta_k}(f_v(J_{k}))\}_{v \in V(\mathcal{N}_{0, k})}$ is identical to $\mathcal{N}_{0, k}$. 
	Then $\mathfrak{U}_k$ is a finite open covering of $J$. 
	We can choose $\delta_{k+1}$ so that the open covering $\mathfrak{U}_{k+1}$ is a refinement of $\mathfrak{U}_{k}$. 
	Since the refinement map agrees with the simplicial map $\phi_{0, k}$, we have $\varprojlim_{k}{H}_q(\mathcal{N}_{0, k}) = \varprojlim_{k}{H}_q(N(\mathfrak{U}_k))$. 
    Without loss of generality, we may assume $\delta_k \to 0$ as $k\to\infty$. 

	What we need to show is that $\{\mathfrak{U}_{k}\}_{k=1}^{\infty}$ is cofinal with respect to refinement. 
	Let $\mathfrak{U}$ be an arbitrary finite open covering of $J$. 
	By the Lebesgue covering lemma,
	there exists $\epsilon > 0$ such that every subset $A$ of $J$ with diameter less than $\epsilon$ is contained in some $U \in \mathfrak{U}$. 
	Since the Lipschitz constants are bounded above by $c < 1$,  
	we have $\diam(f_{v}(J_{k})) \leq c^{k} \diam(J_k) \leq c^{k} \diam(X)$ for every $v \in V(\mathcal{N}_{0, k})$. 
    Since $\delta_k \to 0$, we also have $\diam\big(J \cap  N_{\delta_k}(f_v(J_{k}))\big) \to 0$ as $k\to\infty$. 
	Thus, there exists $k_0 \geq 1$ such that for every $A \in \mathfrak{U}_{k_0}$, we have $\diam(A) < \epsilon$.
	This shows that $\mathfrak{U}_{k_0}$ is a refinement of $\mathfrak{U}$, which completes the proof of the theorem.
\end{proof}

Note that the \v{C}ech homology group $\check{H}_q(X)$ is isomorphic to the singular homology group ${H}_q(X)$ if the topological space $X$ is homeomorphic to the geometric realization of some simplicial complex.  
However, these two groups are not isomorphic in general. 
See Section 6-6 of Wallace \cite{W70}. 

\begin{rem}\label{rem:CohomCoeff}
    Dually, we can define the \v{C}ech-Sumi cohomology as the direct limit of the cohomology groups. 
    Then, there is an isomorphism 
    $\varinjlim {H}^q(\mathcal{N}_{0, k}) \cong \check{H}^q(J),$
    where the right-hand side is the \v{C}ech cohomology group. 
    Also, in general, we can define the \v{C}ech-Sumi (co)homology groups with coefficients in any abelian group. 
    The proof shows that the \v{C}ech-Sumi homology and cohomology are isomorphic to the \v{C}ech homology and cohomology, respectively, with coefficients in any abelian group.
\end{rem}

\section{Connectedness, total disconnectedness, and  local connectedness}\label{sec:conn}
Using the nerves, we determine the (path-)connected components of the limit set $J$, 
and give a sufficient condition for $J$ to be totally disconnected. 

\begin{definition}\label{def:component}
	For a simplicial complex $\mathcal{K}$, denote by $V(\mathcal{K})$ the set of all vertices.  
	Define the equivalence relation $\sim$ on $V(\mathcal{K})$ generated by declaring $u \sim v$ whenever $\{u, v\}$ is a $1$-simplex of $\mathcal{K}$. 

	With the notation above, let $[v] := \{ u \in V(\mathcal{K}) \colon u \sim v\}$ for  $v \in V(\mathcal{K})$.
	For every $v \in V(\mathcal{K})$, we define a subcomplex by $\mathcal{K}_{v} =\{ s \in \mathcal{K} \colon s \subset [v]\}$, 
	which is called a component of $\mathcal{K}$. 
	We denote by $\Con(\mathcal{K})$ the set of all components of $\mathcal{K}$. 
	A simplicial complex $\mathcal{K}$ is said to be connected if $\# \Con(\mathcal{K}) = 1$. 
\end{definition}

For example, consider the nerve $\mathcal{N}_{j, k}$  ($0 \leq j < k$) for a non-autonomous IFS as in Definition \ref{def:CechSumi}. 
Then $V(\mathcal{N}_{j,k}) = I^{(j+1)} \times \dots \times  I^{(k)}$. 
For $u, v \in V(\mathcal{N}_{j,k})$, we have $u \sim v$ in $\mathcal{N}_{j,k}$ if there exist $n \in \NN$ and $v_{0}, v_{1}, \dots, v_{n} \in V(\mathcal{N}_{j,k})$ such that 
$u=v_{0}$, $v=v_{n}$, and $f_{v_{i}}(J_{k}) \cap f_{v_{i+1}}(J_{k}) \neq \emptyset$ for every $i = 0, \dots, n-1$. 

Our definition of component is consistent with that used in Spanier's book \cite[p.138]{Spa}, 
though the notation of the relation $\sim$ is different.  
As shown in the book, for a simplicial complex $\mathcal{K}$, 
if $\mathcal{C} \in \Con(\mathcal{K})$, then the geometric realization $|\mathcal{C}|$ is a path-connected component of $|\mathcal{K}|$. 

\begin{lemma}\label{lem:compIsMapedToComp}
	Let $\mathcal{K}$ and $\mathcal{L}$ be simplicial complexes, and let $\psi  \colon \mathcal{K} \to \mathcal{L}$ be a simplicial map. 
	If $u, v \in V(\mathcal{K})$ satisfy $u \sim v$ in $\mathcal{K}$, 
    then $\psi(u) \sim \psi(v)$ in $\mathcal{L}$. 
\end{lemma}

\begin{proof}
	For every $1$-simplex $s = \{v_0, v_1\} \in \mathcal{K}$, the map $\psi$ maps $s$ either to a $1$-simplex or to the $0$-simplex $\{\psi(v_0)\} = \{\psi(v_1)\}$. 
	Therefore, if $u \sim v$, then $\psi(u) \sim \psi(v)$. 
\end{proof}

Consider now the nerves $\mathcal{N}_{j, k}$, $\mathcal{N}_{j, k+1}$ and the simplicial map $\phi$ defined in Definition \ref{def:CechSumi}. 

\begin{lemma}\label{lem:simplicialSurjective}
	The simplicial map $\phi  \colon \mathcal{N}_{j,k+1} \to \mathcal{N}_{j,k}$ is surjective. 
    More precisely, for every $q \geq 0$ and $q$-simplex $s \in \mathcal{N}_{j,k}$, 
    there exists $q$-simplex $s' \in \mathcal{N}_{j,k+1}$ such that $\phi(s') = s$.  
\end{lemma}

\begin{proof}
	For every $q$-simplex $\{v_0, v_1, \dots, v_q\} \in \mathcal{N}_{j, k}$, 
	we have $\cap_{p=0}^{q} f_{v_p}(J_{k}) \neq \emptyset$ by definition. 
	It follows from $J_{k} = \bigcup_{i \in I^{(k+1)}} f_i^{(k+1)}(J_{k+1})$ that 
	$$\bigcap_{p=0}^{q} \bigcup_{i\in I^{(k+1)}}f_{v_p} \circ f_{i}(J_{k+1}) \neq \emptyset.$$ 
	Therefore, for every $p =0, \dots, q$, there exists $i_p \in I^{(k+1)}$ such that 
	$$\bigcap_{p=0}^q f_{v_p i_p}(J_{k+1}) \neq \emptyset.$$ 
	This shows that $s' = \{v_0i_0, \dots, v_qi_q\} \in \mathcal{N}_{j, k+1}$ and hence $\phi$ is surjective. 
	Note that $s'$ is a $q$-simplex since the concatenated symbols $v_0i_0, \dots, v_qi_q$ are mutually distinct. 
\end{proof}

\begin{lemma}\label{lem:ComponentsSurjective}
	The simplicial map $\phi  \colon \mathcal{N}_{j,k+1} \to \mathcal{N}_{j,k}$ induces a map $\phi_*  \colon \Con(\mathcal{N}_{j,k+1}) \to \Con(\mathcal{N}_{j,k})$, 
	and this is surjective. 
\end{lemma}

\begin{proof}
	Every component of $\mathcal{N}_{j,k+1}$ is of the form $\mathcal{K}_{v} =\{ s \in \mathcal{N}_{j,k+1} \colon s \subset [v]\}$ for some $v \in V(\mathcal{N}_{j,k+1})$. 
	If $u \in V(\mathcal{N}_{j,k+1})$ satisfies $u \sim v$, then $\mathcal{K}_{u} = \mathcal{K}_{v}$ since the equivalence classes are the same $[u]=[v]$. 
	It follows from Lemma \ref{lem:compIsMapedToComp} that $\mathcal{K}_{\phi(u)} = \mathcal{K}_{\phi(v)}$. 
	This shows that $\phi$ induces a map $\phi_*  \colon \Con(\mathcal{N}_{j,k+1}) \to \Con(\mathcal{N}_{j,k})$ so that $\phi_*(\mathcal{K}_{v}) = \mathcal{K}_{\phi(v)}$.
	The surjectivity is due to Lemma \ref{lem:simplicialSurjective}.
\end{proof}

\begin{definition}\label{def:componentTop}
	For a topological space $Y$, we denote by $\Con(Y)$ the set of all connected components of $Y$. 
    Also, we denote by $\pCon(Y)$ the set of all path-connected components of $Y$. 
\end{definition}

The former half of Lemma \ref{lem:ComponentsSurjective} implies that 
the sequence of induced maps $\{\phi_{*}  \colon \Con(\mathcal{N}_{0,k+1}) \to \Con(\mathcal{N}_{0,k})\}_{k > 0}$ forms an inverse system of sets. 
Taking the inverse limit, we derive the following isomorphisms of sets. 

\begin{theorem}\label{th:ConToInverselimit}
	There is a bijection $\Con(J) \to \varprojlim_{k} \Con(\mathcal{N}_{0,k})$. 
	More explicitly, each $C \in \Con(J)$ 
	is mapped to $\mathcal{C}_{k} \in \Con(\mathcal{N}_{0,k})$ such that 
	$$C \subset \bigcup_{v \in V(\mathcal{C}_{k})} f_{v}(J_{k})$$ for every $k > 0.$
\end{theorem}

\begin{proof}    
	For every $C \in \Con(J)$ and for every $k > 0$, 
	we have $C \subset J  = \bigcup_{v \in V(\mathcal{N}_{0,k})} f_{v}(J_{k})$ 
	by Lemma \ref{lem:Huchinson}. 
	Define $V' = \{v \in V(\mathcal{N}_{0,k}) \colon C\cap f_{v}(J_{k}) \neq \emptyset \}$, 
	then $C \subset  \bigcup_{v \in V'} f_{v}(J_{k})$. 
	For every $u, v \in V'$, we have $u \sim v$ since $C$ is connected. 
	Define $\mathcal{C}_{k} \in \Con(\mathcal{N}_{0,k})$ as the component defined by $V'$. 
	Note that if $u \notin V(\mathcal{C}_{k})$, then $C\cap f_{u}(J_{k}) = \emptyset$ by the construction. 

	For $\mathcal{C}_{k} \in \Con(\mathcal{N}_{0,k})$ and $\mathcal{C}_{k+1} \in \Con(\mathcal{N}_{0,k+1})$ defined as above, 
	we show $\phi_*(\mathcal{C}_{k+1}) = \mathcal{C}_{k}$. 
	For $\tilde{v} \in V(\mathcal{C}_{k+1})$, we have $C\cap f_{\tilde{v}}(J_{k+1}) \neq \emptyset$. 
	It follows from Lemma \ref{lem:Huchinson} that $J_{k} =  \bigcup_{i \in I^{(k+1)}} f^{(k+1)}_{i}(J_{k+1})$. 
	Thus, $C\cap f_{\phi(\tilde{v})}(J_{k}) \supset C\cap f_{\tilde{v}}(J_{k+1})\neq \emptyset$. 
	This implies that $\phi(\tilde{v}) \in V(\mathcal{C}_{k})$, and hence $\phi_*(\mathcal{C}_{k+1}) = \mathcal{C}_{k}$. 

	By the universality of inverse limit, 
	the maps $\Con(J) \to  \Con(\mathcal{N}_{0,k})$ induce the unique map $\Con(J) \to  \varprojlim\Con(\mathcal{N}_{0,k})$. 
	We now show that this is injective and surjective. 

	Suppose that $C, C' \in \Con(J)$ satisfy $C \neq C'$. 
    Suppose that $C$ and $C'$ are mapped to $\mathcal{C}_{k}, \mathcal{C}'_{k} \in \Con(\mathcal{N}_{0,k})$ respectively for each $k > 0$. 
	By the construction, 
	$V(\mathcal{C}_{k}) = \{v \in V(\mathcal{N}_{0,k}) \colon C\cap f_{v}(J_{k}) \neq \emptyset \}$. 
	Since $C$ and $C'$ are disjoint compact
    subsets of the metric space $(X, d_X)$, 
	we have 
	$$d_X(C, C') = \min\{d_X(x, x')\colon x \in C, x' \in C'\} > 0.$$ 
	Take a large $k > 0$ so that $c^{k}\diam(X) < d_X(C, C')$, 
	where $c < 1$ is the uniform upper bound of the Lipschitz constants as in Definition \ref{def:NIFS}. 
	Then for every $v \in V(\mathcal{N}_{0,k})$, we have 
	$$\diam(f_{v}(J_{k})) \leq \Lip(f_{v})\diam(J_{k})\leq c^{k}\diam(X) < d_X(C, C').$$
	For this $k$, there does not exist $v \in V(\mathcal{N}_{0,k})$ 
	such that $C\cap f_{v}(J_{k}) \neq \emptyset$ and $C'\cap f_{v}(J_{k}) \neq \emptyset$. 
	Thus,  $V(\mathcal{C}_{k})\cap V(\mathcal{C}'_{k}) = \emptyset$, and hence $\mathcal{C}_{k} \neq \mathcal{C}'_{k}$. 
	This shows that $\Con(J) \to  \varprojlim\Con(\mathcal{N}_{0, k})$ is injective. 

	Fix $(\mathcal{K}_{k})_{k > 0}$ such that $\phi_*(\mathcal{K}_{k+1}) = \mathcal{K}_{k}$ and $\mathcal{K}_{k} \in \Con(\mathcal{N}_{0, k})$ for every $k > 0$. 
	Take $v_{1} \in V(\mathcal{K}_{1})$.
	Then, by Lemma~\ref{lem:simplicialSurjective}, there exists  $v_{k+1} \in V(\mathcal{K}_{k+1})$ such that $\phi(v_{k+1}) = v_{k}$ for every $k > 0$. 
	Since $V(\mathcal{N}_{0, k}) = \prod_{j=1}^{k} I^{(j)}$, 
	the sequence $(v_{k})_{k > 0}$ defines an element $v_{\infty} \in \prod_{j=1}^{\infty} I^{(j)}$. 
	Then, through the coding map $\Pi  \colon \prod_{j=1}^{\infty} I^{(j)} \to X$, we get a point $\Pi(v_{\infty}) \in J$. 
	Let $C$ be the connected component which contains $\Pi(v_{\infty})$. 
	For every $k > 0$, by the definition of the coding map, we have  $\Pi(v_{\infty}) \in   f_{v_{k}}(J_{k}) \subset \bigcup_{v \in V(\mathcal{K}_{k})} f_{v}(J_{k})$. 
	In general, we have $J = \bigcup_{v \in V(\mathcal{N}_{0, k})} f_{v}(J_{k})$, 
	and since  $C$ is a connected subset of $J$, we have $C \subset \bigcup_{v \in V(\mathcal{K}_{k})} f_{v}(J_{k})$. 
	This shows that $\Con(J) \to  \varprojlim\Con(\mathcal{N}_{0, k})$ is surjective, which completes the proof. 
\end{proof}

An application of Theorem~\ref{th:ConToInverselimit} is related to total disconnectedness. 
The following Corollary \ref{cor:totDisc} is one of our main theorems. 

\begin{theorem}\label{th:D}
	Let $c < 1$ be the uniform upper bound on the Lipschitz constants. 
	Fix $C \in \Con(J)$, and let $\mathcal{C}_{k} \in \Con(\mathcal{N}_{0, k})$ as defined in Theorem \ref{th:ConToInverselimit}. 
	If $\lim_{{k \to \infty}} c^{k} \#V(\mathcal{C}_{k}) = 0$, then $\diam(C) = 0$. 
\end{theorem}

\begin{proof}
	By definition, $C \subset \bigcup_{v \in V(\mathcal{C}_{k})} f_{v}(J_{k})$  for every $k > 0$. 
	Since $\mathcal{C}_{k}$ is a component of the simplicial complex $\mathcal{N}_{0,  k}$, the triangle inequality implies 
	$$\diam \left( \bigcup_{v \in V(\mathcal{C}_{k})} f_{v}(J_{k}) \right) 
	\leq \sum_{v \in V({\mathcal{C}_{k})}} \diam(f_{v}(J_{k})). $$
	Thus, 
	$$\diam(C) \leq c^{k} \diam(X)\cdot  \#V(\mathcal{C}_{k}).$$
	If  the right-hand side tends to $0$ as {$k \to \infty$}, then $\diam(C) = 0$. 
	This completes the proof. 
\end{proof}

\begin{cor}\label{cor:totDisc}
	If $\lim_{k \to \infty} c^{k} \max\{\#V(\mathcal{K}) \colon \mathcal{K} \in \Con(\mathcal{N}_{0, k})\} = 0$, then the limit set $J$ is totally disconnected. 
\end{cor}

We can also derive the following corollary of Theorem~\ref{th:ConToInverselimit}, which is the non-autonomous version of Hata's theorem \cite[Theorem 4.6]{Hata}. 

\begin{cor}\label{cor:Hata}
	The following are equivalent.
	\begin{enumerate}
		\item The limit set $J$ is connected. 
		\item For every $k > 0$ and for every $u, v \in  \prod_{j=1}^{k} I^{(j)}$, 
		there exist $n \in \NN$ and $v_{0}, v_{1}, \dots, v_{n} \in \prod_{j=1}^{k} I^{(j)}$ such that 
		$u=v_{0}$, $v=v_{n}$, and $f_{v_{i}}(J_{k}) \cap f_{v_{i+1}}(J_{k}) \neq \emptyset$ for every $i = 0, \dots, n-1$. 
	\end{enumerate}
\end{cor}
    
For an autonomous IFS, the limit set is path-connected if and only if it is connected \cite{Hata}. 
However, the following proposition shows that the limit set may be connected without being either path-connected or locally connected.
    
\begin{prop}\label{prop:TopSineCurve}
 There exists a non-autonomous IFS $(\Phi^{(j)})_{j=1}^\infty$ such that the limit set is homeomorphic to the topologist's sine curve.
\end{prop}

\begin{proof}
    For every $j > 1,$ let $\Phi^{(j)}$ be the collection of the following four maps on $X= [0, 1] \times [0, 1]$. 
    \begin{align*}
        &(x, y) \mapsto (x/3, y/2), \ 
        (x, y) \mapsto ((x+2)/3, y/2), \\
        &(x, y) \mapsto (x/3, (y+1)/2), \ 
        (x, y) \mapsto ((x+2)/3, (y+1)/2).
    \end{align*}
    Then $J_j = C \times[0, 1]$, the product of the middle-thirds Cantor set $C$ and the vertical unit interval, for every $j \geq 1$. 
    Let $\Phi^{(1)} = \{g\}$ be the set consisting of the following map $g\colon X\to X$.

    For every $n\geq0$, set
    \[
    r_n=3^{-n},
    \qquad
    \epsilon_n=
    \begin{cases}
    0,& n \text{ is even},\\
    1,& n \text{ is odd},
    \end{cases}
    \qquad
    p_n=(r_n,\epsilon_n) \in X,
    \]
    and define a line
    $\gamma_n\colon[0,1]\to X$ by
    \[
    \gamma_n(t)
    =
    \begin{cases}
    \bigl((1-t)r_n+tr_{n+1},t\bigr),
    & n \text{ is even},\\[1mm]
    \bigl(tr_n+(1-t)r_{n+1},t\bigr),
    & n \text{ is odd}.
    \end{cases}
    \]
    Let $S_n=\mathrm{Im} \gamma_n$ 
    and define
    \[
    T
    =
    \bigl(\{0\}\times[0,1]\bigr)
    \cup
    \bigcup_{n\geq0}S_n.
    \]
    Thus, $T$ is homeomorphic to the topologist's sine curve.

\begin{center}
\begin{tikzpicture}[x=10cm,y=3cm,
  mid arrow/.style={
    postaction={decorate},
    decoration={
      markings,
      mark=at position 0.5 with {\arrow{>}}
    }
  }
]
\pgfmathsetmacro{\rzero}{1}
\pgfmathsetmacro{\rone}{1/3}
\pgfmathsetmacro{\rtwo}{1/9}
\pgfmathsetmacro{\rthree}{1/27}
\pgfmathsetmacro{\rfour}{1/81}
\pgfmathsetmacro{\tworone}{2/3}
\pgfmathsetmacro{\twortwo}{2/9}

\coordinate (p0) at (\rzero,0);
\coordinate (p1) at (\rone,1);
\coordinate (p2) at (\rtwo,0);
\coordinate (p3) at (\rthree,1);
\coordinate (p4) at (\rfour,0);

\fill[gray!100] (0,0) rectangle (\rfour,1);

\draw[very thick, mid arrow]
  (p0) --
  node[midway,below left] {$\gamma_0$}
  (p1);

\draw[very thick, mid arrow]
  (p2) --
  node[midway,above left] {$\gamma_1$}
  (p1);
  
\draw[very thick]
  (0,0) -- (0,1)
  (p2) -- (p3)
       -- (p4);

\node[left] at (0,0.5) {$T$};


\foreach \x in {0,\rtwo,\twortwo,\rone,\tworone,\rzero}
  \draw (\x,0.02) -- (\x,-0.02);

\node[below] at (0,-0.02) {$0$};
\node[below] at (\rtwo,-0.02) {$r_2$};
\node[below] at (\twortwo,-0.02) {$2r_2$};
\node[below] at (\rone,-0.02) {$r_1$};
\node[below] at (\tworone,-0.02) {$2r_1$};
\node[below] at (\rzero,-0.02) {$r_0$};

\draw[decorate,decoration={brace,mirror,amplitude=4pt}]
  (\tworone,-0.02) -- (\rzero,-0.02)
  node[midway,yshift=-10pt] {$I_0$};

\draw[decorate,decoration={brace,mirror,amplitude=4pt}]
  (\rone,-0.02) -- (\tworone,-0.02)
  node[midway,yshift=-10pt] {$G_0$};

\draw[decorate,decoration={brace,mirror,amplitude=4pt}]
  (\twortwo,-0.2) -- (\rone,-0.2)
  node[midway,yshift=-10pt] {$I_1$};

\draw[decorate,decoration={brace,mirror,amplitude=4pt}]
  (\rtwo,-0.2) -- (\twortwo,-0.2)
  node[midway,yshift=-10pt] {$G_1$};
\end{tikzpicture}
\end{center} 

We decompose $[0,1]$ as follows. For every $n\geq0$, 
let $G_n=[r_{n+1},2r_{n+1}]$ and $I_n=[2r_{n+1},r_n]$.
Then
\[
[0,1]
=
\{0\}
\cup
\bigcup_{n\geq0}(G_n\cup I_n)
\quad
\text{and}
\quad
C\setminus\{0\}
=
\bigcup_{n\geq0}(C\cap I_n).
\]
    We now define a map
    $
    h\colon [0,1]^2\to [0,1]^2.
    $
    On each rectangle $I_n\times[0,1]$, define
    $h(x,t)=\gamma_n(t).$
    We next define $h$ on $G_n\times[0,1]$.
    Write
    $\gamma_n(t)
    =
    \bigl(\gamma_{n,1}(t),t\bigr).$
    Set
    $h(x,t)=\bigl(h_1(x,t),h_2(x,t)\bigr)$
    and define
    $h_2(x,t)=t$
    for every $(x,t)\in G_n\times[0,1]$.
    For $x\in G_n=[r_{n+1},2r_{n+1}]$, define $h_1$ as the linear interpolation between $\gamma_{n+1}$ and $\gamma_n$;
    \[
    h_1(x,t)
    =
    \frac{2r_{n+1}-x}{r_{n+1}}\gamma_{n+1,1}(t)
    +
    \frac{x-r_{n+1}}{r_{n+1}}\gamma_{n,1}(t).
    \]
    Then
    $h(r_{n+1},t)=\gamma_{n+1}(t)$
     and 
    $h(2r_{n+1},t)=\gamma_n(t)$
    for every $t\in[0,1]$.
    Finally, define
    $h(0,t)=(0,t)$
    for every $t\in[0,1]$.
    The definitions agree on the boundaries of the rectangles.
    Moreover, since $\gamma_n(t)\to(0,t)$ uniformly in $t$,
    the map $h$ is continuous on $\{0\}\times[0,1]$.
    Hence $h$ is continuous on $[0,1]^2$.

    We claim that $h$ is Lipschitz. 
    A direct calculation on each of the two rectangles gives 
    \[
    \left|
    \frac{\partial h_1}{\partial x}
    \right|
    =0 
    \text{ and }
    \left|
    \frac{\partial h_1}{\partial t}
    \right|
    \leq \frac23
    \text{ on } I_n\times[0,1]
    \quad 
    \text{ and }
    \quad 
    \left|
    \frac{\partial h_1}{\partial x}
    \right|
    \leq\frac83
    \text{ and }
    \left|
    \frac{\partial h_1}{\partial t}
    \right|
    \leq\frac23
    \text{ on } G_n\times[0,1]
    \]
    Since $h_2(x,t)=t$, it follows that
    \[
    \operatorname{Lip}(h)
    \leq
    \sqrt{
    \left(\frac83\right)^2
    +
    \left(\frac23\right)^2
    +
    1
    }
    =
    \frac{\sqrt{77}}3.
    \]
    
    Define
    $
    g(z)=h(z)/3.
    $
    Then
    $g([0,1]^2)\subset[0,1/3]^2\subset [0,1]^2$
    and
    $\Lip(g)
    \leq
    \frac{\sqrt{77}}{9}
    <1.$
    Thus, $g$ is a contraction on $X$.
    Moreover, since $h(C\times [0,1])=T$, the limit set of the non-autonomous IFS is
    $J=g(J_1)=T/3$
    which is a homeomorphic copy of $T$. 
\end{proof}

In the proof of Proposition~\ref{prop:TopSineCurve}, the limit set $J$ is not locally connected at $(0,t)\in J$ for every $t$. 
Remark that for every $k\geq1$, the nerve
$\mathcal{N}_{k,k+1}$ is disconnected since
$\mathcal{N}_{k,k+1}$ consists of two $1$-simplexes. 
Remark also that the map $g$ is not injective. 

We now consider local connectedness. 

\begin{lemma}
\label{loccon1}
    For each $x\in J$ and $k > 0,$ let \[J_{k, x}:=\bigcup_{v\in V(\mathcal{N}_{0, k}), x\in f_{v}(J_k)} f_{v}(J_{k}).\] Then $\{J_{k, x}\}_{k=1}^{\infty}$ is a neighborhood basis at $x$ in the relative topology of $J$. 
\end{lemma}

\begin{proof}
    For each $x\in J$ and $k > 0,$ we have \[x\in J\setminus \bigcup_{v\in V(\mathcal{N}_{0, k}), x\notin f_{v}(J_{k})} f_{v}(J_{k})\subset J_{k, x}.\] Since the set $$\bigcup_{v\in V(\mathcal{N}_{0, k}), x\notin f_{v}(J_{k})} f_{v}(J_{k})$$ is a finite union of compact subsets of $J$, the complement 
    is an open subset of $J.$ Furthermore, for each $x\in J$ and $k > 0,$
    \[{\rm diam}(J_{k, x})\le 2 \max_{v\in V(\mathcal{N}_{0, k})}{\rm diam}(f_{v}(J_{k}))\le 2 c^{k} {\rm diam}(X),\]
    which tends to $0$ as $k\to \infty.$ Hence $\{J_{k, x}\}_{k=1}^{\infty}$ is a neighborhood basis at $x.$
\end{proof}

\begin{prop}
\label{prop:localConn}
    If $J_k$ is connected for infinitely many $k > 0,$ 
    then $J$ is locally connected.
\end{prop}

\begin{proof}
    Fix $x\in J$ arbitrarily. 
    For each $k > 0,$ let $J_{k, x}$ be the set defined in Lemma \ref{loccon1} such that $\{J_{k, x}\}_{k=1}^{\infty}$ is a neighborhood basis at $x$. 
    If $J_k$ is connected, then the set $J_{k, x}$ is connected. 
    By the assumption, we can choose a strictly increasing sequence $(k_n)_{n\in\NN}$ such that $J_{k_n}$ is connected. 
    Then $\{J_{{k_n}, x}\}_{n\in\NN}$ is a connected neighborhood basis at $x.$ Since $x$ is an arbitrary point, $J$ is locally connected.
\end{proof}

To show the path-connectedness of $J$, we demonstrate the recursive structure of nerves in the next section.

\section{The recursive structure and the subcomplex }\label{sec:exact}
In this section, 
we investigate the recursive structure of a non-autonomous IFS and express it as a subcomplex. 
Then we formulate the exact sequences of homology groups.
This enables us to calculate the ranks of the homology groups in the next section. 

\begin{definition}\label{def:xi}
	For $0 \leq j < k < \ell$, consider three nerves $\mathcal{N}_{j, k}$, $\mathcal{N}_{k, \ell}$, and $\mathcal{N}_{j, \ell}$.
	For every $u = (i_{j+1}, \dots, i_{k})\in V(\mathcal{N}_{j, k})$, 
	define a map $\xi_u \colon V(\mathcal{N}_{k, \ell}) \to V(\mathcal{N}_{j, \ell})$ 
	by $\xi_u(v) = (i_{j+1}, \dots, i_{k}, i_{k+1}, \dots, i_{\ell})$ for $v = (i_{k+1}, \dots, i_{\ell})  \in V(\mathcal{N}_{k, \ell})$. 
\end{definition}

\begin{lemma}\label{lem:xiSimplicial}
    For $0 \leq j < k < \ell$,  
    the map $\xi_u$ is injective and the map 
    $\xi_u  \colon \mathcal{N}_{k, \ell} \to \mathcal{N}_{j, \ell}$ defined by $s \mapsto \xi_u(s)$ is simplicial. 
    Thus,  the image $\xi_u(\mathcal{N}_{k, \ell})$ is a subcomplex of $\mathcal{N}_{j, \ell}$.
\end{lemma}

\begin{proof}
    This follows directly from the definitions.
\end{proof}

\begin{lemma}\label{lem:induction}
    Let $k > 0$ and consider the nerves $\mathcal{N}_{0, k}$,  $\mathcal{N}_{k, k+1}$, 
    and $\mathcal{N}_{0, k+1}$.
    If $\mathcal{N}_{k, k+1}$ is connected, then $\phi_* \colon \Con(\mathcal{N}_{0, k+1}) \to \Con(\mathcal{N}_{0, k})$ is bijective. 
\end{lemma}

\begin{proof}
    Surjectivity is due to Lemma~ \ref{lem:ComponentsSurjective}. 
	Take $\tilde{u}, \tilde{v} \in V(\mathcal{N}_{0, k+1})$ such that $\phi(\tilde{u}) \sim \phi(\tilde{v})$ in $\mathcal{N}_{0, k}$, 
    and we show $\tilde{u} \sim \tilde{v}$ in $\mathcal{N}_{0, k+1}$. 
	Let $u = \phi(\tilde{u})$ and $v = \phi(\tilde{v})$ be the images under the simplicial map $\phi \colon \mathcal{N}_{0, k+1} \to \mathcal{N}_{0, k}$.
	Since $u \sim v$ in $\mathcal{N}_{0, k}$,  
	there exist $n \in \NN$ and $v_{0}, v_{1}, \dots, v_{n} \in V(\mathcal{N}_{0, k})$ such that $\{v_p, v_{p+1}\} \in \mathcal{N}_{0, k}$ for every $p = 0, \dots, n-1$, and $u=v_{0}$, $v=v_{n}$. 
	By Lemma \ref{lem:simplicialSurjective}, for every $p = 0, \dots, n-1$, there exist $i_p$ and $i'_p \in V(\mathcal{N}_{k, k+1})$ such that $\{v_{p}i_p, v_{p+1}i'_{p}\} \in \mathcal{N}_{0, k+1}$.
    Let $i_{n}:=i'_{n-1}$.
	Since $\mathcal{N}_{k, k+1}$ is connected, we have $i'_p \sim i_{p+1}$ in $\mathcal{N}_{k, k+1}$ for every $p = 0, \dots, n-1$, and hence
	$v_{p+1}i'_p = \xi_{v_{p+1}}(i'_{p})\sim \xi_{v_{p+1}}(i_{p+1}) = v_{p+1}i_{p+1}$ in $\mathcal{N}_{0, k+1}$ by Lemma \ref{lem:compIsMapedToComp}. 
	Combining these relations, we have $v_{p}i_p \sim v_{p+1}i'_p\sim v_{p+1}i_{p+1}$ for every $p = 0, \dots, n-1$ in $\mathcal{N}_{0, k+1}$.
	Similarly, Lemma \ref{lem:compIsMapedToComp} also implies that $\tilde{u} \sim v_0i_0$ and $\tilde{v} \sim v_{n}i_n$ in $\mathcal{N}_{0, k+1}$ since $\mathcal{N}_{k, k+1}$ is connected. 
	Therefore, we have 
	$$\tilde{u}  \sim v_0i_0 \sim v_1i_1 \sim \dots \sim v_{n}i_n \sim \tilde{v}$$
	 in $\mathcal{N}_{0, k+1}$.
	This shows that $\phi_* \colon \Con(\mathcal{N}_{0, k+1}) \to \Con(\mathcal{N}_{0, k})$ is injective.
\end{proof}

The converse of Lemma \ref{lem:induction} does not hold. 
Namely,  
in Example \ref{ex:twoGenerator}, the nerves $\mathcal{N}_{0, 1}$ and $\mathcal{N}_{0, 2}$ are connected, 
while the nerve $\mathcal{N}_{1, 2}$ is not connected. 

\begin{cor}\label{cor:SuffConn}
    Suppose that $\mathcal{N}_{k, k+1}$ is connected for every $k \geq 0$. 
    Then, $J$ is connected. 
\end{cor}

\begin{proof}
    By Lemma~\ref{lem:induction}, all the maps 
    $\dots\to\Con(\mathcal{N}_{0, 2}) \to \Con(\mathcal{N}_{0, 1})$
    are bijective. 
    Therefore, the inverse limit is isomorphic to $\Con(\mathcal{N}_{0, 1})$ which consists of only one element.
    It follows from Theorem~\ref{th:ConToInverselimit} that $J$ is connected. 
\end{proof}

\begin{cor}\label{cor:CompIsPathCon}
    Suppose that $\mathcal{N}_{k, k+1}$ is connected for every $k > 0$. 
    Then, every connected component $C\in\Con(J)$ is locally connected and path-connected. 
\end{cor}

\begin{proof}
    By Corollary~\ref{cor:SuffConn}, the set $J_k$ is connected for every $k > 0$. 
    Thus, by Proposition~\ref{prop:localConn}, $J$ is locally connected. 
    It follows from the Hahn-Mazurkiewicz theorem \cite[p.129]{HY} that the connected and locally connected compact metric space $C$ is a continuous image of the unit interval, and hence is path-connected. 
    One can also give a direct proof by constructing a desired path using the nested structure. 
\end{proof}

We obtain a monotonicity for the rank of the first homology. 

\begin{lemma}\label{lem:H1surj}
    Let $k > 0$. If $\mathcal{N}_{k, k+1}$ is connected, 
    then $\phi_* \colon H_1(\mathcal{N}_{0, k+1}) \to H_1(\mathcal{N}_{0, k})$ is surjective. 
\end{lemma}

\begin{proof}
    We denote by $C_1(\mathcal{N})$ the first oriented chain group for a simplicial complex $\mathcal{N}$, that is, 
    the free abelian group generated by the oriented $1$-simplexes $[u, v]$ of $\mathcal{N}$, see \cite[Chapter 4]{Spa}.
    
    Take an arbitrary homology class of $H_1(\mathcal{N}_{0, k})$, and fix a representative $z \in C_1(\mathcal{N}_{0, k})$. 
    We write 
    $z = \sum_{p=0}^n \sigma_p$ where $n \in \NN$ and $\sigma_p$ is an oriented $1$-simplex for every $p = 0, \dots, n$. 
    Since $\partial z = 0$, 
    we may decompose $z$ into a finite sum of cycles each of which is represented by a cyclic sequence of oriented $1$-simplexes. 
    Thus, we may assume $\sigma_p = [v_p, v_{p+1}]$ and $v_{n+1} = v_0$ without loss of generality. 
    By Lemma \ref{lem:simplicialSurjective}, 
    for every $p = 0, \dots, n$, 
    there exists an oriented $1$-simplex $\tilde{\sigma}_p = [\tilde{v}_p, \tilde{v}'_p]\in C_1(\mathcal{N}_{0, k+1})$ such that 
    $\phi_*(\tilde{\sigma}_p) = \sigma_p$. 
    Then  
    $\phi(\tilde{v}'_p) = v_{p+1} = \phi(\tilde{v}_{p+1})$ for every $p = 0, \dots, n-1$.  
    Denote $\tilde{v}'_p= v_{p+1}i'_p $ and $\tilde{v}_{p+1} = v_{p+1}i_{p+1}$ for every $p = 0, \dots, n-1$.   
    Since $\mathcal{N}_{k, k+1}$ is connected, we have $i'_p \sim i_{p+1}$ in $\mathcal{N}_{k, k+1}$, and hence
    $\tilde{v}'_p = \xi_{v_{p+1}}(i'_{p})\sim \xi_{v_{p+1}}(i_{p+1}) = \tilde{v}_{p+1}$ in $\mathcal{N}_{0, k+1}$ for every $p = 0, \dots, n-1$ by Lemma~\ref{lem:compIsMapedToComp}. 
    In particular, for every $p=0, \dots, n-1$, 
    there exists $\tau_p \in C_1(\mathcal{N}_{0, k+1})$ such that 
    $\partial \tau_p =  [\tilde{v}_{p+1}] - [\tilde{v}'_p]$ and $\phi_*(\tau_p) = 0 \in C_1(\mathcal{N}_{0, k})$. 
    Also, there exists $\tau_n \in C_1(\mathcal{N}_{0, k+1})$ such that 
    $\partial \tau_n =  [\tilde{v}_{0}] - [\tilde{v}'_n]$ and $\phi_*(\tau_n) = 0 \in C_1(\mathcal{N}_{0, k})$.  
    We then have $$\tilde{z}:=\sum_{p=0}^n (\tilde{\sigma}_p+ \tau_p) \in \ker (\partial\colon C_1(\mathcal{N}_{0, k+1}) \to C_0(\mathcal{N}_{0, k+1}))$$ and $\phi_*(\tilde{z}) = z$, which shows the surjectivity of $\phi_* \colon H_1(\mathcal{N}_{0, k+1}) \to H_1(\mathcal{N}_{0, k})$. 
\end{proof}

By Lemma~\ref{lem:xiSimplicial}, we can define a subcomplex which represents the recursive structure of the system. 

\begin{definition}\label{def:subcomplex}
	For $0 \leq j < k < \ell$, define $\mathcal{M}_{j, k, \ell} = \bigcup_{u \in V(\mathcal{N}_{j, k})} \xi_u(\mathcal{N}_{k, \ell})$. 
    We call this the $(j,k,\ell)$-subcomplex. 
\end{definition}

\begin{lemma}\label{lem:homologyOfMjkl}
    For $0 \leq j < k < \ell$ and $q \geq 0$, 
    there is an isomorphism 
    $\bigoplus_{u \in V(\mathcal{N}_{j, k})} H_q({\mathcal{N}_{k, \ell}}) \cong H_q(\mathcal{M}_{j, k, \ell})$ 
     induced by the direct sum of the maps $(\xi_u)_* \colon H_q({\mathcal{N}_{k, \ell}}) \to H_q(\mathcal{M}_{j, k, \ell})$. 
\end{lemma}

\begin{proof}
    Since the union $\bigcup_{u \in V(\mathcal{N}_{j, k})} \xi_u(\mathcal{N}_{k, \ell})$ is disjoint, the claim follows from additivity of homology over disjoint unions.
\end{proof}

To calculate the homology group $H_q$ with $q \geq 1$, 
we use the long exact sequence of homology groups.
See \cite[Theorem 4.4]{Spa}, or see  \cite[Theorem 2.13]{Hatcher} and the comment on good pairs below Theorem 2.13 of the book. 

\begin{theorem}
	\label{th:exact}
	Let $0 \leq j < k < \ell$. 
	Consider the nerve $\mathcal{N}_{j, \ell}$ and the $(j, k, \ell)$-subcomplex $\mathcal{M}_{j, k, \ell}$. 
	Then there is a long exact sequence of the homology groups
	\begin{equation*}
	\begin{tikzcd}[column sep=1em]
	\cdots \arrow[r, "\partial"] & {H}_q(\mathcal{M}_{j, k, \ell}) \arrow[r, "\iota"] & {H}_q(\mathcal{N}_{j, \ell}) 	\arrow[r, "\varpi"] & {H}_q(\mathcal{N}_{j, \ell}, \mathcal{M}_{j, k, \ell})
	\arrow[r, "\partial"] & {H}_{q-1}(\mathcal{M}_{j, k, \ell}) \arrow[r] & \cdots
	\end{tikzcd}
	\end{equation*}
	where 
	$\iota$ is induced by inclusion and 
	$\varpi$ is induced by the quotient map. 
\end{theorem}

In the next section, we directly calculate the relative homology group ${H}_q(\mathcal{N}_{j, \ell}, \mathcal{M}_{j, k, \ell})$, although it is sometimes better to use the fact that it is isomorphic to the reduced homology group $\tilde{H}_q(|\mathcal{N}_{j, \ell}|/|\mathcal{M}_{j, k, \ell}|)$ of the quotient space of geometric realization $|\mathcal{N}_{j, \ell}|$ by $|\mathcal{M}_{j, k, \ell}|$ for all $q \geq 0$.

\section{Non-autonomous fractal squares and their generalizations}\label{sec:ex}
In this section, we apply Theorems \ref{res:CechSumi}, \ref{res:totDisc}, \ref{res:Hata}, and \ref{res:exact} to the so-called fractal squares and to their higher-dimensional generalizations. 
Throughout this section, we consider the following class. 

\begin{definition}\label{ex:fracSq}
    Let $d \in \NN$ and $X = [0, 1]^{d}$, the $d$-dimensional unit  cube. 
    For each $k =1, 2, \dots, d$, 
    let $n_{k} \in \NN$ with $n_{k} \geq 2$ and set $I = \prod_{k = 1}^{d}\{0, 1, \dots, n_{k}-1\}$. 
    For each $\mathbf{i} = (i_{1}, \dots, i_{d}) \in I$, 
    define a contractive map $f_{\mathbf{i}} \colon X \to X$ by 
    $$f_{\mathbf{i}}(x)= f_{\mathbf{i}}(x_{1}, \dots, x_{d})   = \left(\frac{x_{1} +i_{1}}{n_{1}}, \dots, \frac{x_{d} +i_{d}}{n_{d}}\right).$$
    For every $j \geq 1$, a non-empty subset $I^{(j)} \subset I$ is given. 
    Then $\Phi^{(j)} = \{f_{\mathbf{i}}\}_{\mathbf{i}\in I^{(j)}}$ forms a non-autonomous IFS 
    $(\Phi^{(j)})_{j \geq 1}$. 
\end{definition}

For instance, 
suppose that  $d = 2$ and $I^{(j)} = I^{(j+1)}$ for every $j \geq 1$. 
If $n_{1} = n_{2}$, then the (autonomous) limit set is known as a fractal square. 
The Sierpi\'nski carpet is an example of a fractal square. 
If $n_{1} \neq n_{2}$, then the limit set is called a Bedford-McMullen carpet. 

Our interest lies in the case where the sets $I^{(j)}$ need not be identical. 
We shall investigate when the limit set is connected, totally disconnected, or has other topological properties. 
We also compute the first Betti numbers.

\subsection{Connectedness}
\begin{lemma}\label{lem:corePlane}
    Suppose that for some $k =1, 2, \dots, d$, for every $j \geq 1$,
    we have $\#\left(I\setminus I^{(j)}\right) < n_{k}$. 
    Then there exists a non-empty set $\tilde{J} \subset [0, 1]$ 
    such that $J \supset [0, 1]^{k-1} \times \tilde{J} \times  [0, 1]^{d-k}$. 
    Here, if $k =1$ or $d$, then the set $[0, 1]^{k-1} \times \tilde{J} \times  [0, 1]^{d-k}$ is understood as 
    $\tilde{J} \times  [0, 1]^{d-1}$ or 
    $[0, 1]^{d-1} \times\tilde{J}$, respectively. 
\end{lemma}

\begin{proof}
    Let $\pi_{k} \colon \RR^{d} \to \RR$ be the projection onto the $k$th coordinate. 
    For every $j \geq 1$, since $\#\left(I\setminus I^{(j)}\right) < n_{k}$, 
    the image $\pi_{k}\left(I\setminus I^{(j)}\right)$ 
    has a non-empty complement $\tilde{I}^{(j)} := \{0, 1, \dots, n_{k}-1\} \setminus \pi_{k}\left(I\setminus I^{(j)}\right)$. 
    We now consider the one-dimensional maps
    $g_{i}(x) = \left(x +i\right)/{n_{k}}$;
    then $(\{ g_{i}\}_{i \in \tilde{I}^{(j)}})_{j \geq 1}$ is a non-autonomous IFS on $[0, 1]$. 
    Let $\tilde{J}$ denote its limit set, which is a non-empty subset of $[0, 1]$. 
    
    We show that $J \supset [0, 1]^{k-1} \times \tilde{J} \times  [0, 1]^{d-k}$. 
    For every $(x_{1}, \dots, x_{d}) \in [0, 1]^{k-1} \times \tilde{J} \times  [0, 1]^{d-k}$, 
    we consider an $n_{\ell}$-ary expansion of $x_{\ell}$. 
    That is, for every $\ell = 1, 2, \dots, d$,  and for every $j \geq 1$, 
    there exists $i^{(j)}_{\ell} \in \{0, 1, \dots, n_{\ell}-1\} $ such that 
    $$x_{\ell} = \sum_{j=1}^{\infty}\frac{i^{(j)}_{\ell}}{n_{\ell}^{j}}. $$ 
    Since $x_{k} \in \tilde{J}$, we can assume that $i^{(j)}_{k} \in \tilde{I}^{(j)}$ for every $j \geq 1$. 
    Then, for every $j \geq 1$, we have $\mathbf{i}^{(j)} = (i^{(j)}_{1}, \dots, i^{(j)}_{d}) \in I^{(j)}$. 
    Therefore, the sequence $(\mathbf{i}^{(1)}, \mathbf{i}^{(2)}, \dots )   \in \prod_{j=1}^{\infty} I^{(j)}$ 
    is mapped to the point $(x_{1}, \dots, x_{d})$ under the coding map of the non-autonomous IFS $(\Phi^{(j)})_{j=1}^{\infty}$.
    This completes the proof. 
\end{proof}

By a similar argument, under some condition, we can show that the limit set contains a line. 

\begin{lemma}\label{lem:coreLine}
    Suppose that for some $k =1, 2, \dots, d$, for every $j \geq 1$, 
    we have $\#\left(I\setminus I^{(j)}\right) < \prod_{\ell \neq k} n_{\ell}$. 
    Then for every $\ell \neq k$, there exists  $x^{*}_{ \ell} \in [0, 1]$ such that 
    $$J \supset
     \{x^{*}_{1}\} \times \dots \times \{x^{*}_{k-1}\} \times [0, 1] \times 
    \{x^{*}_{k+1}\} \times \dots \times \{x^{*}_{d}\}.$$ 
\end{lemma}

In Lemma \ref{lem:coreLine},  the right-hand side is understood as in Lemma~\ref{lem:corePlane} if $k =1$ or $d$. 
Moreover, if $d =1$, then  $\prod_{\ell \neq k} n_{\ell}$ is understood as $1$. 

\begin{proof}
    Let $\hat{\pi}_{k} \colon \RR^{d} \to \RR^{d-1}$ be the projection which deletes the $k$th coordinate, namely 
    $$\hat{\pi}_{k}
    (i_{1}, \dots, i_{k-1}, i_{k}, i_{k+1}, \dots, i_{d}) = 
    (i_{1}, \dots, i_{k-1},        i_{k+1}, \dots, i_{d}).$$ 
    For every $j \geq 1$, since $\#\left(I\setminus I^{(j)}\right) < \prod_{\ell \neq k} n_{\ell}$, 
    the image $\hat{\pi}_{k}\left(I\setminus I^{(j)}\right)$ 
    has the non-empty complement $\hat{I}^{(j)} := \left( \prod_{\ell \neq k} \{0, 1, \dots, n_{\ell}-1\} \right)\setminus \hat{\pi}_{k}\left(I\setminus I^{(j)}\right)$. 
    Then, consider the $(d-1)$-dimensional non-autonomous IFS generated by $(\hat{I}^{(j)})_{j \geq 1}$.
    Let $\hat J$ denote its non-empty limit set and choose a point $(x_1^*,\ldots,x_{k-1}^*,x_{k+1}^*,\dots,x_d^*)\in\widehat J$. 
    This point has the desired property. 
\end{proof}

To prove the connectedness of the limit set, 
we utilize the following concept. 

\begin{definition}\label{def:adjascent}
    For every $\ell \geq 1$, let $u, v \in I^\ell$. 
    We say that $u$ and $v$ are adjacent if $f_u(X) \cap f_v(X)$ has topological dimension $(d-1)$. 
\end{definition}

For the case of $\ell = 1$, 
two symbols $\mathbf{i}, \mathbf{i}' \in I = \prod_{k = 1}^d \{0, 1, \dots, n_k - 1\}$ are adjacent
if and only if they differ by $1$ in exactly one component and are identical in all other components. 
For instance, 
$\mathbf{i} = (0, 0, \dots, 0)$ and $\mathbf{i}' = (1, 0, \dots, 0)$ are adjacent, 
$\mathbf{i}$ and $\mathbf{i}'' = (0, 1, 0,  \dots, 0)$ are adjacent, but 
$\mathbf{i}'$ and $\mathbf{i}''$ are not adjacent since 
$$f_{\mathbf{i}'}(X) \cap f_{\mathbf{i}''}(X) = \{1/n_1\} \times \{1/n_2\} \times [0,1/n_3]\times \dots \times [0, 1/n_d]$$
is $(d-2)$-dimensional. 

We now show the (local) connectedness of non-autonomous fractal cubes. 

\begin{theorem}\label{th:connectedSq}
    Suppose that  $\#\left(I\setminus I^{(j)}\right) < d$  for every $j \geq 1$. 
    Then the limit set $J$ of $(\Phi^{(j)})_{j \geq 1}$ is connected, locally connected, and path-connected. 
\end{theorem}

\begin{proof}
    We prove that the simplicial complex $\mathcal{N}_{0, 1}$ is connected. 
    The proof also shows that the simplicial complex $\mathcal{N}_{j, j+1}$ is connected for every $j \geq 0$, 
    and it follows from Corollaries~\ref{cor:SuffConn} and \ref{cor:CompIsPathCon} 
    that $J_j$ is connected, locally connected, and path-connected.  
    
    Consider the covering $\{f_{\mathbf{i}}(X)\}_{\mathbf{i} \in I}$ of $X$, 
    and call each set $f_{\mathbf{i}}(X)$ a piece. 
    Then every piece intersects at least $d$ pieces. 
    For example, if $\mathbf{0} = (0, \dots, 0)$, 
    then the piece $f_{\mathbf{0}}(X)$ intersects each piece $f_{\mathbf{i}}(X)$ for which $\mathbf{0}$ and $\mathbf{i}$ are adjacent. 
    
    Fix $\mathbf{i} \in I^{(1)}$. 
    Since $\#\left(I\setminus I^{(1)}\right) < d$, 
    there exists $\mathbf{i}' \in I^{(1)}$ which is adjacent to $\mathbf{i}$. 
    Let $k \in\{1, 2, \dots, d\}$ be the unique index such that $\mathbf{i}$ and $\mathbf{i}'$ differ in their $k$th coordinate. 
    Since $n_{\ell} \geq 2$ for every $\ell = 1, 2, \dots, d$, we have 
    $\#\left(I\setminus I^{(j)}\right) < d  \leq \prod_{\ell \neq k} n_{\ell}$ for every $j \geq 2$. 
    It follows from  Lemma \ref{lem:coreLine} that for every $\ell \neq k$, there exist $x^{*}_{ \ell} \in [0, 1]$ 
    such that $J_{1} \supset  \{x^{*}_{1}\} \times \dots \times \{x^{*}_{k-1}\} \times [0, 1] \times 
    \{x^{*}_{k+1}\} \times \dots \times \{x^{*}_{d}\}$. 
    Then the images $\pi_{k} \circ f_{\mathbf{i}}(J_{1})$ and $\pi_{k} \circ f_{\mathbf{i}'}(J_{1})$, under the $k$th coordinate projection $\pi_{k}$, are consecutive intervals of length $1/n_{k}$. 
    In particular, $f_{\mathbf{i}}(J_{1}) \cap f_{\mathbf{i}'}(J_{1}) \neq \emptyset$. 
    
    Since adjacent pieces of $\{f_{\mathbf{i}}(J_{1})\} _{\mathbf{i} \in I^{(1)}}$ intersect in a chain-like manner, 
    the simplicial complex $\mathcal{N}_{0, 1}$ is connected.
    This completes the proof. 
\end{proof}

The assumption $\#\left(I\setminus I^{(j)}\right) < d$ is sharp in the almost-sure setting when $d=2$, as will be shown in Subsection~\ref{ssec:hom}.

\subsection{Total disconnectedness}
Corollary \ref{cor:totDisc} is often helpful in proving total disconnectedness, 
but for the fractals of Definition \ref{ex:fracSq}, 
we prefer to use a more straightforward argument to show that the limit set is totally disconnected. 

\begin{definition}\label{def:cut}
    Let $I' \subset I$. 
    We say that $I'$ has a cut normal to the $x_{k}$-axis if 
    there exists $i_{k} \in \{0, 1, \dots, n_{k} -1 \}$ such that  $i_{k} \notin \pi_{k}(I').$
    Here, ${\pi}_{k} \colon \RR^{d} \to \RR$ is the projection onto the $k$th coordinate. 
\end{definition}

\begin{lemma}\label{lem:cut}
    Suppose that there exists $j \geq 1$ such that $I^{(j)}$ has a cut normal to the $x_{k}$-axis for some $k =1, 2, \dots, d$. 
    Then for every $C \in \Con(J)$, we have 
    $$\diam( \pi_{k}(C)) \leq \frac{n_{k} - 1}{n_{k}^{j}}.$$ 
\end{lemma}

\begin{proof}
    If $j = 1$, then the inequality is trivial. 
    Thus, we assume that $I^{(j)}$ has a cut normal to the $x_{k}$-axis for $j \geq 2$. 
    
    Let $\tilde{I}^{(j)} = \pi_{k}(I^{(j)})$ be a subset of $\{0, 1, \dots, n_{k}-1\}$ for every $j \geq 1$. 
    By definition, $\pi_{k}(J)$ is equal to the limit set of one-dimensional non-autonomous IFS generated by $(\tilde{I}^{(j)})_{j \geq 1}$. 
    Thus, if $i_{k} \notin \tilde{I}^{(j)}$, then $\pi_{k}(J)$ does not intersect the open interval 
    $$\left(
    \frac{i_{k}}{n_{k}^{j}}   + \frac{m}{n_{k}^{j-1}}, 
    \frac{i_{k}+1}{n_{k}^{j}} + \frac{m}{n_{k}^{j-1}}
    \right)$$
    for every $m =0, 1, \dots, n_{k}^{j-1}-1$. 
    For every $C \in \Con(J)$, the image $\pi_{k}(C)$ is also a connected subset of $\RR$, 
    and hence $\pi_{k}(C)$ is contained in a closed interval of length at most $(n_{k} - 1)/{n_{k}^{j}}$. 
    This completes the proof. 
\end{proof}

\begin{cor}\label{cor:totDiscSq}
    If $\# \{j \geq 1 \colon I^{(j)} \text{ has a cut normal to the $x_{k}$-axis}\} = \infty$  for every $k =1, 2, \dots, d$, 
    then the limit set $J$ is totally disconnected. 
\end{cor}

\begin{proof}
    For every $C \in \Con(J)$ and $k =1, 2, \dots, d$, we have $\diam( \pi_{k}(C)) = 0$ by Lemma \ref{lem:cut}. 
    Thus, $C$ must be a singleton.  
\end{proof}

Note that the limit set can be totally disconnected even if $I^{(j)}$ does not have cuts for every $j \geq 1$. See \cite{Cristea}. 

\subsection{Connected components of randomly generated non-autonomous fractal cubes}\label{ssec:random}
\begin{definition}\label{def:randomSq}
    Let $r \in \NN$ such that $0 \leq r \leq (\prod_{k=1}^{d}n_{k}) -1$. 
    Denote by $\mathcal{P}_{r}(I)$ the set of all subsets $I'$ of $I$ such that $\#(I\setminus I') = r$. 
\end{definition}

The following theorem provides a partial answer to a non-autonomous higher-dimensional analog of the Mandelbrot percolation problem. See Subsection~\ref{sssec:perc} for the original problem.

\begin{theorem}\label{th:randomSq}
    Let $r \in \NN$ such that $0 \leq r \leq (\prod_{k=1}^{d}n_{k}) -1$. 
    Suppose that each $I^{(j)}$ is randomly chosen independently according to the uniform distribution on $\mathcal{P}_{r}(I)$. 
    Then the limit set $J$ of $(\Phi^{(j)})_{j=1}^{\infty}$ satisfies the following. 
    \begin{enumerate}
        \item If $r < d$, then $J$ is connected and locally connected. 
        \item If there exists $k =1, 2, \dots, d$ such that $r < \prod_{\ell \neq k} n_{\ell}$, 
        then $J$ contains a line segment which connects the face of the unit $d$-cube defined by $x_k = 0$ to the face $x_k = 1$. 
        \item If there exists $k =1, 2, \dots, d$ such that $r \geq \prod_{\ell \neq k} n_{\ell}$, then almost surely, for every $C \in \Con(J)$, we have $\diam( \pi_{k}(C)) = 0$. 
        \item If $r \geq \prod_{\ell \neq k} n_{\ell}$ for every $k =1, 2, \dots, d$, then $J$ is totally disconnected almost surely. 
    \end{enumerate}
\end{theorem}

\begin{proof}
    The first item is a consequence of Theorem~\ref{th:connectedSq}. 
    The second item is due to Lemma~\ref{lem:coreLine}. 
    Suppose that $r \geq \prod_{\ell \neq k} n_{\ell}$. 
    Then, for instance, $I' =\{(i_1, \dots, i_d) \in I \colon i_k \neq 0\}$ has a cut normal to the $x_k$-axis and $\#(I \setminus I') = \prod_{\ell \neq k} n_{\ell}$. 
    Therefore, for every $j \geq 1$, with positive probability, the set $I^{(j)}$ has a cut normal to the $x_k$-axis. 
    The second Borel-Cantelli lemma implies that such events happen infinitely often, and hence $\diam( \pi_{k}(C)) = 0$ by Lemma ~\ref{lem:cut}. 
    Moreover, suppose that $r \geq \prod_{\ell \neq k} n_{\ell}$ for every $k =1, 2, \dots, d$. 
    Then, with probability one, for every $k =1, 2, \dots, d$, the set $I^{(j)}$ has a cut normal to the $x_{k}$-axis infinitely often. 
    It follows from Corollary \ref{cor:totDiscSq} that $J$ is totally disconnected almost surely.
\end{proof}

Theorem \ref{th:randomSq} is in contrast with the $n_1 \times n_2$ Mandelbrot  percolation, whose horizontal critical value is equal to the vertical one \cite{FF}, for the case $d = 2$ and $n_1 \neq n_2$.

\subsection{Homology groups for randomly generated non-autonomous fractal rectangles}\label{ssec:hom}
We now consider the two-dimensional case.
Our main interest lies in the first homology.

\begin{figure}[ht]
    \centering
    \includegraphics[width=0.4\linewidth]{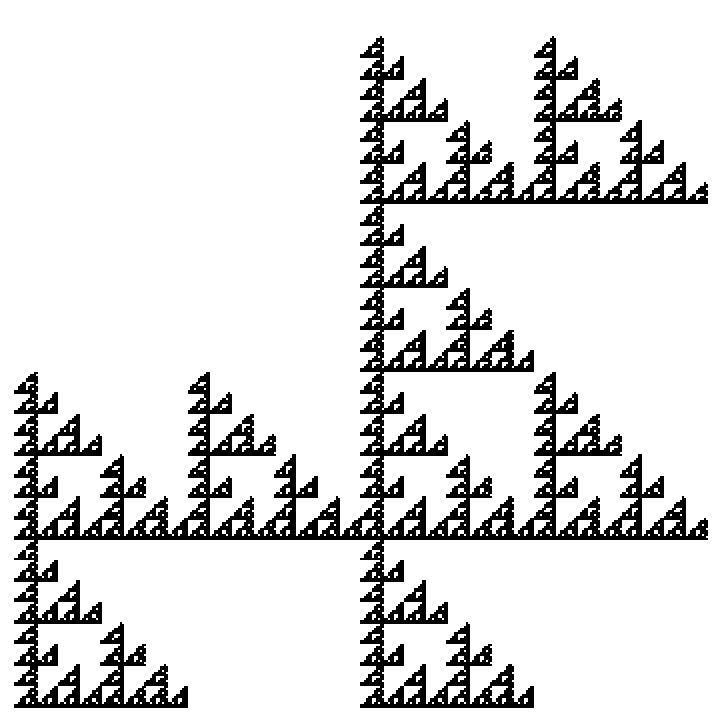}
    \begin{tikzpicture}[scale=1.2]
    	\draw (0,1)--(0,2)--(1, 2);
    	\node (00) at (0,0) [below=18pt, right=12pt] {$\mathcal{N}_{1, 2}$};
    	\coordinate (00) at (0,1) node at (00) [below] {};
    	\coordinate (10) at (0,2) node at (10) [below] {};
        \coordinate (11) at (1,2) node at (11){}; 
    	\fill (00) circle (2pt) (10) circle (2pt) (11) circle (2pt);
        \draw (5, 3)--(5, 2);
        \draw (5, 3)--(6, 3);
        \draw (5, 3)--(5, 2)--(5, 1)--(5, 0);
        \draw (3, 1)--(3+0,0);
        \draw (3, 1)--(4, 1);
        \draw (4, 1)--(4,1)--(5, 1)--(6, 1);
    	\node (00) at (5,0) [below=12pt] {$\mathcal{N}_{0, 2}$};
    	\coordinate (0000) at (3,0) node at (0000) [below] {};
        \coordinate (1000) at (5,0) node at (1000) [below] {};
        \coordinate (0001) at (3,1) node at (0001) [below] {};
        \coordinate (0101) at (4,1) node at (0101) [below] {};
        \coordinate (1001) at (5,1) node at (1001) [below] {};
        \coordinate (1101) at (6,1) node at (1101) [below] {};
        \coordinate (1010) at (5,2) node at (1010) [below] {};
        \coordinate (1011) at (5,3) node at (1011) [below] {};
        \coordinate (1111) at (6,3) node at (1111) [below] {};
    	\fill 
        (0000) circle (2pt) 
        (1000) circle (2pt)
        (0001) circle (2pt) 
        (0101) circle (2pt) 
        (1001) circle (2pt)
        (1101) circle (2pt) 
        (1010) circle (2pt) 
        (1011) circle (2pt)
        (1111) circle (2pt) ;
        \draw (9, 3)--(9, 2);
        \draw (9, 3)--(10, 3);
        \draw (9, 3)--(9, 2);
        \draw[dashed] (9, 2)--(9, 1);
        \draw (9, 1)--(9, 0);
        \draw (7, 1)--(7,0);
        \draw (7, 1)--(8, 1);
        \draw[dashed] (8, 1)--(9, 1);
        \draw (9, 1)--(10, 1);
    	\node (m00) at (9,0) [below=12pt] {$\mathcal{M}_{0, 1, 2}$};
    	\coordinate (m0000) at (7,0) node at (m0000) [below] {};
        \coordinate (m1000) at (9,0) node at (m1000) [below] {};
        \coordinate (m0001) at (7,1) node at (m0001) [below] {};
        \coordinate (m0101) at (8,1) node at (m0101) [below] {};
        \coordinate (m1001) at (9,1) node at (m1001) [below] {};
        \coordinate (m1101) at (10,1) node at (m1101) [below] {};
        \coordinate (m1010) at (9,2) node at (m1010) [below] {};
        \coordinate (m1011) at (9,3) node at (m1011) [below] {};
        \coordinate (m1111) at (10,3) node at (m1111) [below] {};
    	\fill 
        (m0000) circle (2pt) 
        (m1000) circle (2pt)
        (m0001) circle (2pt) 
        (m0101) circle (2pt) 
        (m1001) circle (2pt)
        (m1101) circle (2pt) 
        (m1010) circle (2pt) 
        (m1011) circle (2pt)
        (m1111) circle (2pt) ;
    \end{tikzpicture}
    \caption{
    An example of a fractal square (above) 
    and the nerves and $(0, 1, 2)$-subcomplex (below) where $d =2$, $n_1 = 2$, $n_2 = 2$, and $r=1$. 
    In the bottom-right figure, 
    the dashed lines are elements of $\mathcal{N}_{0, 2}\setminus\mathcal{M}_{0, 1, 2}$ while 
    the solid lines are elements of $\mathcal{M}_{0, 1, 2}$. 
    Observe that each dashed line represents a relative homology class of $H _1(\mathcal{N}_{0, 2}, \mathcal{M}_{0, 1, 2})$. }
    \label{fig:2by2-1}
\end{figure}
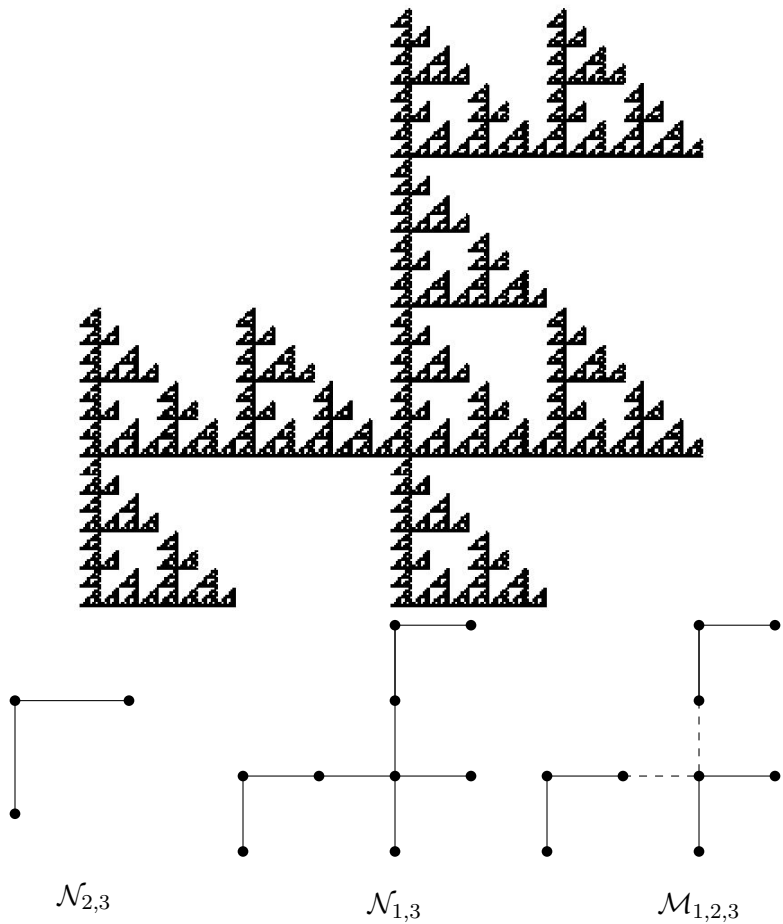

\begin{definition}\label{def:corner}
    Suppose $d = 2$.
    We say that the non-autonomous IFS $(\Phi^{(j)})_{j=1}^{\infty}$ satisfies the no-corner condition if
    for every $j \geq 0$, the limit set $J_j$ of $(\Phi^{(k+j)})_{k=1}^{\infty}$ contains none of the four points $(0, 0)$, $(1, 0)$, $(0, 1)$, $(1, 1)$ of $X = [0, 1]^2$. 
\end{definition}

\begin{lemma}\label{lem:nocorner}
    Suppose $d = 2$ and $r \geq 1$. 
    Suppose that the sets $I^{(j)}$ are chosen independently according to the uniform distribution on $\mathcal{P}_{r}(I)$. 
    Then almost surely the non-autonomous IFS  $(\Phi^{(j)})_{j=1}^{\infty}$ satisfies the no-corner condition. 
\end{lemma}

\begin{proof}
    Since the event that $I^{(j)} \ni (0, 0)$ for every $j \geq 1$ has probability zero, 
    the limit set $J$ does not contain the point $(0, 0) \in [0, 1]^2$ with probability one. 
    Similarly, with probability one, 
    the limit set $J$ does not contain the other corners $(0,1)$, $(1, 0)$, and $(1,1)$. 
    The same holds for $J_j$ for every $j \geq 1$.  
    Thus, with probability one, $(\Phi^{(j)})_{j=1}^{\infty}$ satisfies the no-corner condition. 
\end{proof}

We now calculate the growth rate of the rank of $H_0$.

\begin{theorem}\label{th:H0Growth}
    Suppose that $d=2$ and $2\le r<\min\{n_1,n_2\}.$
    Suppose that each \(I^{(j)}\) is chosen independently according to
    the uniform distribution on $\mathcal P_r(I)$.
    Then, almost surely,
    \[
    \lim_{k\to\infty}
    \frac{1}{k}
    \log\operatorname{rank}H_0(\mathcal N_{0,k})
    =
    \log (n_1n_2-r).
    \]
\end{theorem}

The proof uses an elementary renewal argument.

\begin{proof}
    By Lemma~\ref{lem:nocorner}, we may restrict attention to
    realizations satisfying the no-corner condition.
    
    We consider the event that a specific configuration happens. 
    For every \(s\ge1\), let \(B_s\) be the event that
    \[
    \begin{aligned}
     &(0,0)\in I^{(2s-1)},\qquad 
     (1,0),(0,1)\notin I^{(2s-1)},\\
     &(n_1-1,n_2-1)\in I^{(2s)},\quad
     (n_1-2,n_2-1),(n_1-1,n_2-2)\notin I^{(2s)}.
    \end{aligned}
    \]
    The events \((B_s)_{s\ge1}\) are independent and identically
    distributed with positive probability.
    
    Suppose that \(B_s\) occurs and put $a:=(0,0), b:=(n_1-1,n_2-1)\in I$. 
    We claim that, for every $v\in V(\mathcal N_{0,2s-2})$,
    the vertex $vab\in V(\mathcal N_{0,2s})$ is isolated.
    Indeed, the rectangle $f_{vab}(X)$ is located at the upper-right
    corner of $f_{va}(X)$. Its left and lower edge-neighbors inside
    $f_{va}(X)$ correspond respectively to
    \[
    v a(n_1-2,n_2-1)
    \quad\text{and}\quad
    v a(n_1-1,n_2-2),
    \]
    and these words are absent by the definition of \(B_s\). Its right
    and upper edge-neighbors lie respectively inside the rectangles
    corresponding to
    \[
    v(1,0)
    \quad\text{and}\quad
    v(0,1),
    \]
    which are also absent.
    Thus, \(f_{vab}(X)\) has no retained edge-neighbor at level $2s$.
    Any remaining intersection with another level-$2s$ rectangle
    could occur only at a corner. Such an intersection is excluded by
    the no-corner condition. It follows that $vab$ is an isolated
    vertex of $\mathcal N_{0,2s}$.
    
    Since $\#V(\mathcal N_{0,2s-2})=(n_1n_2-r)^{2s-2}$,
    we obtain $\mathrm{rank}H_0(\mathcal N_{0,2s})
    \geq (n_1n_2-r)^{2s-2}$ whenever $B_s$ occurs.
    For $k\geq1$, define
    \[
    T_k
    :=
    \max\left\{
    s\leq \left\lfloor\frac{k}{2}\right\rfloor:B_s\text{ occurs}
    \right\},
    \]
    with the convention that $T_k=0$ if none of the events
    \(B_1,\ldots,B_{\left\lfloor\frac{k}{2}\right\rfloor}\) occurs.
    Put $\tau_k:=2T_k$.
    
    We claim that
    \[
    \frac{\tau_k}{k}\longrightarrow1
    \qquad\text{almost surely}.
    \]
    Indeed, let $p:=\mathbb P(B_1)$>0. 
    For every \(\varepsilon\in(0,1)\), the event $T_k\le (1-\varepsilon)k/2$ 
    implies that none of the events \(B_s\) with $(1-\varepsilon)k/2<s\leq \left\lfloor k/2\right\rfloor$
    occurs. 
    Since the events \(B_s\) are independent, we have
    \[
    \mathbb P\bigl(T_k\le(1-\varepsilon)\frac{k}{2}\bigr)
    \le
    (1-p)^{\lfloor\varepsilon \frac{k}{2}\rfloor}.
    \]
    Therefore,
    \[
    \sum_{k=1}^{\infty}
    \mathbb P\bigl(T_k\le(1-\varepsilon)\frac{k}{2}\bigr)
    <\infty.
    \]
    By the Borel-Cantelli lemma, almost surely, 
    we have $T_k>(1-\varepsilon)k/2$
    for all sufficiently large \(k\). 
    Since $T_k\le k/2$, it follows that
    \[
    \frac{T_k}{k/2}\longrightarrow1
    \qquad\text{almost surely}.
    \]
    Consequently, $\tau_k/k\to 1$ almost surely. 
        
    By Lemma~\ref{lem:ComponentsSurjective}, the rank of $H_0(\mathcal N_{0,k})$ is non-decreasing as $k$ grows.
    Thus, for all sufficiently large \(k\),
    \[
    \mathrm{rank}H_0(\mathcal N_{0,k})
    \geq
    \mathrm{rank}H_0(\mathcal N_{0,\tau_k})
    \geq
    (n_1n_2-r)^{\tau_k-2}.
    \]
    It follows that almost surely
    \[
    \begin{aligned}
    \liminf_{k\to\infty}
    \frac{1}{k}
    \log\operatorname{rank}H_0(\mathcal N_{0,k})
    &\ge
    \liminf_{k\to\infty}
    \frac{\tau_k-2}{k}\log (n_1n_2-r)
    =\log (n_1n_2-r).
    \end{aligned}
    \]
    
    On the other hand, 
    \[
    \operatorname{rank}H_0(\mathcal N_{0,k})
    \le
    \#V(\mathcal N_{0,k})
    =
    (n_1n_2-r)^k,
    \]
    and hence
    \[
    \limsup_{k\to\infty}
    \frac{1}{k}
    \log\operatorname{rank}H_0(\mathcal N_{0,k})
    \le
    \log (n_1n_2-r).
    \]
    Combining the two inequalities gives
    \[
    \lim_{k\to\infty}
    \frac{1}{k}
    \log\operatorname{rank}H_0(\mathcal N_{0,k})
    =
    \log(n_1n_2-r)
    \]
    almost surely.
\end{proof}

\begin{lemma}\label{lem:H2vanish}
    Suppose $d = 2$ and that $(\Phi^{(j)})_{j=1}^{\infty}$ satisfies the no-corner condition. 
    Then, for every $0 \leq j < k$ and $q \geq 2$, the nerve $\mathcal{N}_{j,k}$ contains no $q$-simplex. 
\end{lemma}

\begin{proof}
    Suppose
    $$A:= f_{\mathbf{i}_0}(J_1) \cap f_{\mathbf{i}_1}(J_1) \cap f_{\mathbf{i}_2}(J_1) \neq \emptyset$$ for mutually distinct $\mathbf{i}_0, \mathbf{i}_1, \mathbf{i}_2$.  
    Since  $X = \bigcup_{\mathbf{i} \in I}f_{\mathbf{i}}(X)$ and $J_1 \subset X$, 
    the intersection $A$ is a singleton, say $\{x\}$, and $f_{\mathbf{i}_0}^{-1}(x)
    \in J_1$ must be one of the four corners $(0, 0)$, $(1, 0)$, $(0, 1)$, or $(1, 1)$. 
    This contradicts the assumption that $(\Phi^{(j)})_{j=1}^{\infty}$ satisfies the no-corner condition. 
    Thus, the nerve $\mathcal{N}_{0,1}$ contains no $q$-simplex for every $q \geq 2$.
    A similar argument shows that $\mathcal{N}_{j,k}$ contains no $q$-simplex for every $q \geq 2$ and for every $0 \leq j < k$. 
\end{proof}

\begin{lemma}\label{lem:H1surjSquare}
    Suppose $d = 2$, $\# (I\setminus I^{(j)}) < \min\{n_{1}, n_{2}\}$ for every $j \geq 1$, and 
    that $(\Phi^{(j)})_{j=1}^{\infty}$ satisfies the no-corner condition. 
    Then $\phi_* \colon H_1(\mathcal{N}_{j, k+1}) \to H_1(\mathcal{N}_{j, k})$ is surjective for every $0 \leq j < k$. 
\end{lemma}

The proof is similar to that of Lemma~\ref{lem:H1surj}.

\begin{proof}
    We may assume $j=0$.
    Take an arbitrary homology class of $H_1(\mathcal{N}_{0, k})$, and fix a representative $z \in C_1(\mathcal{N}_{0, k})$. 
    Since $\partial z = 0$, 
    without loss of generality, we can write $z = \sum_{p=0}^n \sigma_p$ where $n \in \NN$, $v_{n+1} = v_0$, and 
    $\sigma_p = [v_p, v_{p+1}]$ is an oriented $1$-simplex for every $p = 0, \dots, n$. 
    
    For every $p = 0, \dots, n$, we shall construct 
    a suitable oriented $1$-simplex $\tilde{\sigma}_p\in C_1(\mathcal{N}_{0, k+1})$ such that 
    $\phi_*(\tilde{\sigma}_p) = \sigma_p$. 
    Since $\{v_p, v_{p+1}\} \in\mathcal{N}_{0,k}$, 
    we have $f_{v_{p}}(J_{k}) \cap f_{v_{p+1}}(J_{k}) \neq \emptyset$. 
    By the no-corner condition, the vertices $v_p$ and $v_{p+1}$ are adjacent. 
    More precisely, $f_{v_{p}}(X)$ is a rectangle with width $1/n_1^{k}$ and height $1/n_2^{k}$, and 
    rectangles $f_{v_{p}}(X)$ and $f_{v_{p+1}}(X)$ are placed side by side, either horizontally or vertically.
    Consider, for instance, the case where they are arranged horizontally.
    Since $\# (I\setminus I^{(j)}) < n_2$ for every $j \geq 1$, 
    it follows from Lemma \ref{lem:coreLine} that $J_{k+1}$ contains a horizontal line segment $L$ of length $1$. 
    This $L$ determines $i_p$ and $i'_p \in V(\mathcal{N}_{k, k+1})$ such that 
    $f_{v_{p}i_p}(X)$ and $f_{v_{p+1}i'_p}(X)$ are arranged horizontally, 
    and hence $\{v_{p}i_p, v_{p+1}i'_p\}\in\mathcal{N}_{0,k+1}$. 
    Set $\tilde{\sigma}_p = [v_{p}i_p, v_{p+1}i'_p]\in C_1(\mathcal{N}_{0, k+1})$.
    Then $\phi_*(\tilde{\sigma}_p) = \sigma_p$. 
    Note that, regardless of whether the configuration is vertical or horizontal, there exists a component $\mathcal{C} \in\Con(\mathcal{N}_{k, k+1})$ such that for every $p=0, \dots, n$, the vertices satisfy $i_p\in \mathcal{C}$ and $i'_p \in \mathcal{C}$. 
    
    We next find $\tau_p \in C_1(\mathcal{N}_{0, k+1})$ such that 
    $\partial \tau_p =  [v_{p+1}i_{p+1}] - [v_{p+1}i'_p]$, where we define $i_{n+1}:=i_0$. 
    By the construction of $i_p$ and $i'_p$,  we can join $i_{p+1}$ and $i'_p$ in $\mathcal{C}$. 
    Since $\xi_{v_{p+1}} \colon\mathcal{N}_{k,k+1} \to \mathcal{N}_{0,k+1}$ is simplicial, it induces the desired $\tau_p \in C_1(\mathcal{N}_{0, k+1})$. 
    Then,  $\phi_*(\tau_p) = 0 \in C_1(\mathcal{N}_{0, k})$. 
    
    We then have $$\tilde{z}:=\sum_{p=0}^n (\tilde{\sigma}_p+ \tau_p) \in \ker (\partial\colon C_1(\mathcal{N}_{0, k+1}) \to C_0(\mathcal{N}_{0, k+1}))$$ and $\phi_*(\tilde{z}) = z$, which shows the surjectivity of $\phi_* \colon H_1(\mathcal{N}_{0, k+1}) \to H_1(\mathcal{N}_{0, k})$. 
\end{proof}

By the same Borel-Cantelli argument as in the proof of
Theorem~\ref{th:H0Growth}, we have the following growth of $H_1$. 

\begin{theorem}\label{th:H1growthInf}
    Suppose that $d = 2$ and $1 \leq r < \min\{n_{1}, n_{2}\}$. 
    Suppose that $n_1 > 2$ or $n_2 > 2$.
    Suppose that each $I^{(j)}$ is randomly chosen independently according to the uniform distribution on $\mathcal{P}_{r}(I)$. 
    Then almost surely 
    $$\liminf_{k\to\infty}\frac{1}{k}\log\mathrm{rank} {H}_1(\mathcal{N}_{0,k}) \geq \log (n_1n_2 - r).$$
\end{theorem}

\begin{proof}
    By Lemma \ref{lem:nocorner}, with probability one, $(\Phi^{(j)})_{j=1}^{\infty}$ satisfies the no-corner condition, and hence 
    we work on the probability-one event on which the no-corner condition holds.
    For every $j \geq 0$ and $u, v\in V(\mathcal{N}_{j, j+1})$, 
    we have that $\{u, v\}$ is a $1$-simplex of $\mathcal{N}_{j, j+1}$ if and only if $u$ and $v$ are adjacent. 
    Therefore, the random variables $\{\mathrm{rank} H_1(\mathcal{N}_{j, j+1})\}_{j=0}^\infty$ are independent. 
    
    Since either $n_1>2$ or $n_2>2$, the probability of the event 
    $\mathrm{rank}H_1(\mathcal{N}_{j,j+1})\geq 1$ is positive. 
    For every \(k\geq2\), define
    \[
    \tau_k
    :=
    \max\{0\leq j\leq k-1 : \mathrm{rank}H_1(\mathcal{N}_{j,j+1})\geq 1\},
    \]
    whenever the set on the right-hand side is nonempty.
    By the same Borel-Cantelli argument as in the proof of
    Theorem~\ref{th:H0Growth}, almost surely \(\tau_k\) is defined for
    all sufficiently large \(k\) and
    \[
    \frac{\tau_k}{k}\longrightarrow 1
    \qquad\text{as }k\to\infty.
    \]
    
    Since \(\tau_k+1\leq k\), repeated application of
    Lemma~\ref{lem:H1surjSquare} gives
    \[
    \operatorname{rank}H_1(\mathcal{N}_{0,k})
    \geq
    \operatorname{rank}H_1(\mathcal{N}_{0,\tau_k+1}).
    \]
    We next estimate the right-hand side. By
    Theorem~\ref{th:exact} and Lemma~\ref{lem:H2vanish}, for every
    \(0\leq j\) and \(j+1<\ell\), there is an exact sequence
    $
    0
    \to
    H_1(\mathcal{M}_{j,j+1,\ell})
    \to
    H_1(\mathcal{N}_{j,\ell}).
    $
    Together with Lemma~\ref{lem:homologyOfMjkl}, this implies that
    \[
    \operatorname{rank}H_1(\mathcal{N}_{j,\ell})
    \geq
    \#I^{(j+1)}
    \operatorname{rank}H_1(\mathcal{N}_{j+1,\ell}).
    \]
    Applying this inequality successively with
    \(\ell=\tau_k+1\), we obtain
    \[
    \begin{aligned}
    \operatorname{rank}H_1(\mathcal{N}_{0,\tau_k+1})
    &\geq
    \left(
    \prod_{j=1}^{\tau_k}\#I^{(j)}
    \right)
    \operatorname{rank}H_1(
    \mathcal{N}_{\tau_k,\tau_k+1}
    )\\
    &=
    (n_1n_2-r)^{\tau_k}
    \operatorname{rank}H_1(
    \mathcal{N}_{\tau_k,\tau_k+1}
    ).
    \end{aligned}
    \]
    By the definition of \(\tau_k\), we have
    $\operatorname{rank}H_1(
    \mathcal{N}_{\tau_k,\tau_k+1}
    )\geq1.$
    Consequently,
    $\operatorname{rank}H_1(\mathcal{N}_{0,k})
    \geq
    (n_1n_2-r)^{\tau_k}.$
    Therefore,
    \[
    \begin{aligned}
    \liminf_{k\to\infty}
    \frac{1}{k}
    \log\operatorname{rank}H_1(\mathcal{N}_{0,k})
    \geq
    \liminf_{k\to\infty}
    \frac{\tau_k}{k}\log(n_1n_2-r)
    =
    \log(n_1n_2-r).
    \end{aligned}
    \]
\end{proof}

For $0 \leq j < k < \ell$, 
every $w \in V(\mathcal{N}_{j, \ell})$ is uniquely written as the concatenation $w = uv$ of $u \in V(\mathcal{N}_{j, k})$ and $v \in V(\mathcal{N}_{k, \ell})$.

\begin{lemma}\label{lem:relativeC1}
    Suppose $d = 2$ and that $(\Phi^{(j)})_{j=1}^{\infty}$ satisfies the no-corner condition. 
    Let $0 \leq j < k < \ell$. 
     If $u = u' \in V(\mathcal{N}_{j, k})$ and 
        $\{uv, u'v'\}$ is a $1$-simplex of the nerve $\mathcal{N}_{j,\ell}$, then 
        $\{v, v'\}$ is a $1$-simplex of the nerve $\mathcal{N}_{k,\ell}$. 
\end{lemma}

\begin{proof}
    Suppose $u = u'$ and that
    $\{uv, u'v'\}$ is a $1$-simplex of the nerve $\mathcal{N}_{j,\ell}$. 
    Then $$
    f_u(f_v(J_\ell)) \cap f_{u}(f_{v'}(J_\ell)) = 
    f_u(f_v(J_\ell)) \cap f_{u'}(f_{v'}(J_\ell)) \neq \emptyset.$$ 
    Since $f_u$ is injective, we have $f_v(J_\ell) \cap f_{v'}(J_\ell) \neq \emptyset$, 
    which completes the proof. 
\end{proof}

\begin{lemma}\label{lem:adjacent}
    Suppose $d = 2$ and that $(\Phi^{(j)})_{j=1}^{\infty}$ satisfies the no-corner condition. 
    Let $0 \leq j$ and $j+1 < \ell$. 
    Let $\mathbf{i}, \mathbf{i}' \in V(\mathcal{N}_{j, j+1})$ and $v, v' \in V(\mathcal{N}_{j+1, \ell})$. 
    Suppose that $\mathbf{i} \neq \mathbf{i}'$ and $\{\mathbf{i}v, \mathbf{i}'v'\} \in \mathcal{N}_{j, \ell}$. 
    Then $\mathbf{i}$ and $\mathbf{i}'$ are adjacent.    
    Moreover, for every $v'' \in V(\mathcal{N}_{j+1, \ell})$ with $v'' \neq v'$, we have $\{\mathbf{i}v, \mathbf{i}'v''\} \notin \mathcal{N}_{j, \ell}$.
\end{lemma}

\begin{proof}
    Since $\{\mathbf{i}v, \mathbf{i}'v'\} \in \mathcal{N}_{j, \ell}$, 
    we have $f_{\mathbf{i}}f_v(J_\ell) \cap f_{\mathbf{i}'}f_{v'}(J_\ell) \neq \emptyset,$ and hence $f_{\mathbf{i}}f_v(X) \cap f_{\mathbf{i}'}f_{v'}(X) \neq \emptyset.$
    By the no-corner condition, this intersection is not $0$-dimensional. 
    Since $\mathbf{i} \neq \mathbf{i}'$, this intersection is not $2$-dimensional. 
    Hence, $f_{\mathbf{i}}f_v(X) \cap f_{\mathbf{i}'}f_{v'}(X)$ is $1$-dimensional. 
    From the property of rectangles, 
    $\mathbf{i}$ and $\mathbf{i}'$ are adjacent. 

    Let $\mathbf{i} = (i_1, i_2)$ and $\mathbf{i}'= (i'_1, i'_2)$. 
    Without loss of generality, 
    we may assume $i_1 +1 = i'_1$ and $i_2 =i'_2$. 
    Then, since $f_{\mathbf{i}}f_v(X) \cap f_{\mathbf{i}'}f_{v'}(X)$ is $1$-dimensional, 
    $f_v(X)$ and $f_{v'}(X)$ are horizontally arranged rectangles and share a vertical line segment.
    Thus, for every $v'' \in V(\mathcal{N}_{j+1, \ell})$ with $v'' \neq v'$, 
    we have $f_{\mathbf{i}}f_v(J_\ell) \cap f_{\mathbf{i}'}f_{v''}(J_\ell) = \emptyset$ by the no-corner condition.  
\end{proof}

\begin{lemma}\label{lemma:relrativeH1}
    Suppose $d = 2$ and that $(\Phi^{(j)})_{j=1}^{\infty}$ satisfies the no-corner condition. 
    For $0 \leq j$ and $j+1 < \ell$, 
    the relative homology group $H_1(\mathcal{N}_{j, \ell}, \mathcal{M}_{j, j+1, \ell})$ is isomorphic to a free abelian group with basis  
        $$\{ \{\mathbf{i}v, \mathbf{i}'v'\} \in \mathcal{N}_{j, \ell} 
        \colon \mathbf{i} \neq \mathbf{i}' \in V(\mathcal{N}_{j, j+1}) 
        \text{ are adjacent and } v, v'\in V(\mathcal{N}_{j+1, \ell})\}.$$ 
\end{lemma}

\begin{proof}
    Let $C_q(\mathcal{K})$ the $q$th oriented chain group for a simplicial complex $\mathcal{K}$, that is, 
    the free abelian group generated by the oriented $q$-simplexes of $\mathcal{K}$, see \cite[Chapter 4]{Spa}. 
    We denote by $$C_q(\mathcal{N}_{j, \ell}, \mathcal{M}_{j, j+1, \ell}) = C_q(\mathcal{N}_{j, \ell})/C_q(\mathcal{M}_{j, j+1, \ell}).$$ 
    Since $V(\mathcal{N}_{j, \ell}) = V(\mathcal{M}_{j, j+1, \ell})$, 
    we have $C_0(\mathcal{N}_{j, \ell}, \mathcal{M}_{j, j+1, \ell}) = 0$. 
    By Lemma \ref{lem:H2vanish}, we have 
    $$C_2(\mathcal{N}_{j, \ell}) = 0 = 
     C_2(\mathcal{N}_{j, \ell}, \mathcal{M}_{j, j+1, \ell}).$$
    Thus, the relative homology group $H_1(\mathcal{N}_{j, \ell}, \mathcal{M}_{j, j+1, \ell})$ is isomorphic to $C_1(\mathcal{N}_{j, \ell}, \mathcal{M}_{j, j+1, \ell})$
    by definition. 

    For a $1$-simplex $\{w, w'\}\in \mathcal{N}_{j, \ell}$, 
    we write $w = \mathbf{i}v$ and $w' = \mathbf{i}'v'$ where 
    $\mathbf{i}, \mathbf{i}' \in V(\mathcal{N}_{j, j+1})$ and 
    $v, v' \in V(\mathcal{N}_{j+1, \ell})$ respectively. 
    By definition of $\mathcal{M}_{j, j+1, \ell}$, we have that 
    $\mathbf{i} = \mathbf{i}'$ if $\{w, w'\}\in \mathcal{M}_{j, j+1, \ell}$. 
    Conversely, if $\mathbf{i} = \mathbf{i}'$, then the injectivity of $f_\mathbf{i}$ implies 
    $$f_\mathbf{i}(f_v(X) \cap f_{v'}(X)) \supset f_\mathbf{i}(f_v(X))\cap f_{\mathbf{i}'}(f_{v'}(X)) \neq \emptyset,$$
    and hence $\{v, v'\}\in \mathcal{N}_{j+1, \ell}$. 
    This shows that 
    $\mathbf{i} = \mathbf{i}'$ if and only if $\{\mathbf{i}v, \mathbf{i}'v'\}\in \mathcal{M}_{j, j+1, \ell}$. 
    Suppose that $\mathbf{i} \neq \mathbf{i}'$. 
    Then, by Lemma \ref{lem:adjacent},  $\mathbf{i}$ and $\mathbf{i}'$ are adjacent. 
    This shows that 
    $C_1(\mathcal{N}_{j, \ell}, \mathcal{M}_{j, j+1, \ell})$ is the free abelian group with basis 
    $$\{\{\mathbf{i}v, \mathbf{i}'v'\} \in \mathcal{N}_{j,\ell} \colon
        \mathbf{i} \neq \mathbf{i}' \in V(\mathcal{N}_{j, j+1}) \text{ are adjacent and } v, v'\in V(\mathcal{N}_{j+1, \ell})\}.$$
    This completes the proof. 
\end{proof}    

\begin{theorem}\label{th:countingRanks}
    Suppose $d = 2$ and that $(\Phi^{(j)})_{j=1}^{\infty}$ satisfies the no-corner condition. 
    Let $0 \leq j$ and $j+1 < \ell$. Then we have 
    \begin{align*}
        &\mathrm{rank} H_1(\mathcal{N}_{j, \ell}) 
        -\mathrm{rank}{H}_{0}(\mathcal{N}_{j, \ell}) \\
        =& \# I^{(j+1)}\cdot (
        \mathrm{rank} H_1(\mathcal{N}_{j+1, \ell}) -
        \mathrm{rank} H_0(\mathcal{N}_{j+1, \ell}))\\
         &+\#\left\{ \{\mathbf{i}v, \mathbf{i}'v'\} \in \mathcal{N}_{j, \ell} 
        \colon \mathbf{i} \neq \mathbf{i}' \in V(\mathcal{N}_{j, j+1}) \text{ are adjacent and } v, v'\in V(\mathcal{N}_{j+1, \ell})\right\}. 
    \end{align*}
    In particular, 
    \begin{align*}
        &\mathrm{rank} H_1(\mathcal{N}_{0, \ell}) 
        -\mathrm{rank}{H}_{0}(\mathcal{N}_{0, \ell}) 
        \geq 
        \left(\prod_{j=1}^{\ell-1}\#I^{(j)}\right)\cdot 
        \left(
        \mathrm{rank} H_1(\mathcal{N}_{\ell-1, \ell}) -
        \mathrm{rank} H_0(\mathcal{N}_{\ell-1, \ell})
        \right).
    \end{align*}
\end{theorem}

\begin{proof}
    By Theorem \ref{th:exact}, we have the following long exact sequence. 
    \begin{equation*}
	\begin{tikzcd}
	 0 \arrow[r] & {H}_1(\mathcal{M}_{j, j+1, \ell}) \arrow[r] & {H}_1(\mathcal{N}_{j, \ell}) 	\arrow[r] & {H}_1(\mathcal{N}_{j, \ell}, \mathcal{M}_{j, j+1, \ell})
	 \\
    \arrow[r] &{H}_{0}(\mathcal{M}_{j, j+1, \ell}) \arrow[r] & {H}_0(\mathcal{N}_{j, \ell}) 	\arrow[r] & {H}_0(\mathcal{N}_{j, \ell}, \mathcal{M}_{j, j+1, \ell}). 
	\end{tikzcd}
	\end{equation*}
    Since the relative $0$th oriented chain group satisfies ${C}_0(\mathcal{N}_{j, \ell}, \mathcal{M}_{j, j+1, \ell}) = 0$, we have ${H}_0(\mathcal{N}_{j, \ell}, \mathcal{M}_{j, j+1, \ell}) = 0.$ 
    By additivity of ranks in an exact sequence, we have 
    \begin{align*}
         0 = \ 
         &\mathrm{rank}{H}_1(\mathcal{M}_{j, j+1, \ell}) 
        -\mathrm{rank}{H}_1(\mathcal{N}_{j, \ell}) 
        +\mathrm{rank}{H}_1(\mathcal{N}_{j, \ell}, \mathcal{M}_{j, j+1, \ell})\\
        &-\mathrm{rank}{H}_{0}(\mathcal{M}_{j, j+1, \ell})
        +\mathrm{rank}{H}_{0}(\mathcal{N}_{j, \ell}). 
    \end{align*}
    By Lemma \ref{lem:homologyOfMjkl}, we have 
    $$\mathrm{rank}{H}_q(\mathcal{M}_{j, j+1, \ell}) = \#V(\mathcal{N}_{j, j+1}) \cdot\mathrm{rank}{H}_q(\mathcal{N}_{j+1, \ell})$$ 
    for $q = 0$ and $q = 1$. 
    By Lemma \ref{lemma:relrativeH1}, the rank of ${H}_1(\mathcal{N}_{j, \ell}, \mathcal{M}_{j, j+1, \ell})$ is the same as the number of basis elements. 
    This completes the proof of the first assertion. 

    Setting $j=0$, we have 
    \begin{align*}
        \mathrm{rank} H_1(\mathcal{N}_{0, \ell}) -\mathrm{rank}{H}_{0}(\mathcal{N}_{0, \ell})
        \geq & \# I^{(1)}\cdot (
        \mathrm{rank} H_1(\mathcal{N}_{1, \ell}) -
        \mathrm{rank} H_0(\mathcal{N}_{1, \ell})). 
    \end{align*}
    Setting $j=1$, we have 
    \begin{align*}
        \mathrm{rank} H_1(\mathcal{N}_{1, \ell}) -\mathrm{rank}{H}_{0}(\mathcal{N}_{1, \ell})
        \geq & \# I^{(2)}\cdot (
        \mathrm{rank} H_1(\mathcal{N}_{2, \ell}) -
        \mathrm{rank} H_0(\mathcal{N}_{2, \ell})). 
    \end{align*}
    By repeating this procedure, we have the desired inequality. 
    This completes the proof. 
\end{proof}

\begin{lemma}\label{lem:2by2Mjkl}
    Suppose $d=2$, $n_1 = n_2 = 2$, and for every $j \geq 1$,
    we have $\#\left(I\setminus I^{(j)}\right) = 1$. 
    Suppose also that $(\Phi^{(j)})_{j=1}^{\infty}$ satisfies the no-corner condition. 
    Then, for every $0 \leq j$ and $j+1 < \ell$, 
    $$\#\{ \{\mathbf{i}v, \mathbf{i}'v'\} \in \mathcal{N}_{j, \ell} 
    \colon \mathbf{i} \neq \mathbf{i}' \in V(\mathcal{N}_{j, j+1}) \text{ are adjacent and }  v, v'\in V(\mathcal{N}_{j+1, \ell})\} = 2.$$
\end{lemma}

See Figure~\ref{fig:2by2-1} for an intuitive illustration. 

\begin{proof}
    Without loss of generality, we may assume $j = 0$. 
    Fix a $1$-simplex $\{\mathbf{i}^{(1)}, \mathbf{i}'^{(1)}\} \in \mathcal{N}_{0, 1}$. 
    Without loss of generality, 
    we may assume $\mathbf{i}^{(1)} = (i_1, i_2)$, $\mathbf{i}'^{(1)}= (i'_1, i'_2)$,  $i_1 =0$, $i'_1 =1$, and $i_2= i'_2 = 0$. 
    Since $f_\mathbf{i}(X) \cap f_\mathbf{i'}(X) = \{1/2\} \times [0, 1/2]$,  the intersection satisfies 
    $$f_{\mathbf{i}^{(1)}}(J_1) \cap f_{\mathbf{i'^{(1)}}}(J_1) = 
        f_{\mathbf{i}^{(1)}}(J_1 \cap (\{1\} \times [0, 1])) \cap 
        f_{\mathbf{i'^{(1)}}}(J_1 \cap(\{0\} \times [0, 1])).$$  
    By the no-corner condition, we can replace the closed vertical interval by the open vertical interval so that 
    $$f_{\mathbf{i}^{(1)}}(J_1) \cap f_{\mathbf{i'}^{(1)}}(J_1) = 
    f_{\mathbf{i}^{(1)}}(J_1 \cap (\{1\} \times (0, 1))) \cap 
    f_{\mathbf{i'}^{(1)}}(J_1\cap(\{0\} \times (0, 1))).$$  
    Let ${I}^{(2)}_0 := {I}^{(2)} \cap (\{0\} \times \{0, 1\})$ and 
    ${I}^{(2)}_1 := {I}^{(2)} \cap (\{1\} \times \{0, 1\})$. 
    Then, by the property of fractal rectangles, we have 
    $$J_1 \cap (\{0\} \times (0, 1)) = \bigcup_{\mathbf{i} \in I^{(2)}_0}f_{\mathbf{i}} (J_2\cap (\{0\} \times (0, 1)))$$
    and
    $$J_1 \cap (\{1\} \times (0, 1)) = \bigcup_{\mathbf{i} \in I^{(2)}_1}f_{\mathbf{i}} (J_2\cap (\{1\} \times (0, 1))).$$
    Thus, there exist unique
    $\mathbf{i}^{(2)} \in {I}^{(2)}_1$ and 
    $\mathbf{i'}^{(2)} \in {I}^{(2)}_0$ such that
    $$f_{\mathbf{i}^{(1)}}f_{\mathbf{i}^{(2)}}(J_2) \cap f_{\mathbf{i'}^{(1)}}f_{\mathbf{i'}^{(2)}}(J_2) \neq \emptyset.$$  
    By repeating this procedure, there exist unique
    $v, v' \in V(\mathcal{N}_{1, \ell})$ such that
    $\{\mathbf{i}^{(1)}v, \mathbf{i}'^{(1)}v'\} \in \mathcal{N}_{0, \ell} $. 
    By the no-corner condition, the number of $1$-simplexes of $\mathcal{N}_{0, 1}$ is $2$.  
    Thus, we have 
    $$\#\{ \{\mathbf{i}v, \mathbf{i}'v'\} \in \mathcal{N}_{0, \ell} 
    \colon \mathbf{i} \neq \mathbf{i}' \in V(\mathcal{N}_{0, 1}) \text{ are adjacent and }  v, v'\in V(\mathcal{N}_{1, \ell})\} = 2.$$
    This completes the proof. 
\end{proof}

The following theorem provides an answer to a non-autonomous and homological  analog of the Mandelbrot percolation problem.

\begin{theorem}\label{th:randomSqDim2}
    Suppose that $d = 2$ and $1 \leq r \leq n_{1}n_{2} -1$. 
    Suppose that each $I^{(j)}$ is randomly chosen independently according to the uniform distribution on $\mathcal{P}_{r}(I)$. 
    Then almost surely the limit set $J$ satisfies $\check{H}_q(J) = 0$ for every $q \geq 2$. 
    Moreover, we have the following. 
    \begin{enumerate}
        \item If $r = 1,$ then $\check{H}_0(J) \cong \ZZ$. 
            \begin{enumerate}
                \item If $n_1 = n_2 = 2,$ then $\check{H}_1(J) = 0$.
                \item If $(n_1, n_2) \neq (2, 2),$ then 
                    $$\lim_{k \to \infty}\frac{1}{k}\log (\mathrm{rank} H_1(\mathcal{N}_{0, k})) = \log(n_1n_2 - r)$$
                almost surely. 
            \end{enumerate}
        \item If $2 \leq r < \min\{n_{1}, n_{2}\}$,   
        then for each $q=0,1$, 
        $$\lim_{k \to \infty}\frac{1}{k}\log (\mathrm{rank} H_q(\mathcal{N}_{0, k})) = \log(n_1n_2 - r)$$
        almost surely. 
        \item If $n_{1} \leq r < n_{2}$ (resp. $n_{2} \leq r < n_{1}$), then almost surely, every connected component of $J$ is a horizontal (resp. vertical) line segment. One of them is a line segment of length $1$, and the others may possibly degenerate to single points.  
        \item If $r \geq \max\{n_{1}, n_{2}\}$, then $J$ is totally disconnected almost surely. 
    \end{enumerate}
\end{theorem}

\begin{proof}
    The third and fourth items are corollaries of Theorem~\ref{th:randomSq}. 
    We prove the first and the second items. 
    By Lemma~\ref{lem:nocorner}, we may work on the probability-one event on which $(\Phi^{(j)})_{j=1}^{\infty}$ satisfies the no-corner condition in what follows.
    
    By Lemma \ref{lem:H2vanish}, the nerve $\mathcal{N}_{j,k}$ contains no $q$-simplex for every $0 \leq j < k$ and $q \geq 2$. 
    Hence, $\check{H}_q(J) = \varprojlim{H}_q(\mathcal{N}_{0,k}) = 0$ for every $q \geq 2$ by Theorem \ref{th:cechsumi}.
    
    We assume $r = 1$ and consider the $0$th homology. 
    By Theorem \ref{th:connectedSq}, the limit set $J$ is connected and hence $\check{H}_0(J) \cong \ZZ$. 

    We assume $r = 1$ and $n_1 = n_2 =2$ and prove the statement (1a). 
    We show, by backward induction, that $H_1(\mathcal{N}_{j, \ell}) = 0$ for every $0 \leq j$ and $j+1 < \ell$.
    Since $r = 1$ and $n_1 = n_2 =2$, for every $k \geq 0$, the nerve $\mathcal{N}_{k,k+1}$ consists of three $0$-simplexes and two $1$-simplexes, which do not form any $1$-cycle. 
    Thus, $H_1(\mathcal{N}_{k, k+1}) = 0$. 
    Suppose that $1 \leq k < \ell$ and $H_1(\mathcal{N}_{k, \ell}) = 0$, and we show $H_1(\mathcal{N}_{k-1, \ell}) = 0$ by using  Theorem \ref{th:exact}.
    By Lemma \ref{lemma:relrativeH1}, the relative homology group $H_1(\mathcal{N}_{k-1, \ell}, \mathcal{M}_{k-1, k, \ell})$ is the free abelian group with basis
    $$\{\{\mathbf{i}v, \mathbf{i}'v'\} \in \mathcal{N}_{k-1,\ell} \colon
        \mathbf{i} \neq \mathbf{i}' \in V(\mathcal{N}_{k-1, k}) \text{ are adjacent and }  v, v'\in V(\mathcal{N}_{k, \ell})\}.$$
    By Lemma \ref{lem:2by2Mjkl}, this basis has precisely two elements. 
    Also, we have the following exact sequence. 
    \begin{equation*}
	\begin{tikzcd}
	 0 \arrow[r] & H_1(\mathcal{M}_{k-1, k,  \ell}) \arrow[r] & {H}_1(\mathcal{N}_{k-1, \ell}) 	\arrow[r] & {H}_1(\mathcal{N}_{k-1, \ell}, \mathcal{M}_{k-1, k, \ell})
	 \\
    \arrow[r, "\partial"] & H_0(\mathcal{M}_{k-1, k,  \ell}) \arrow[r] & {H}_0(\mathcal{N}_{k-1, \ell}) 	\arrow[r] & {H}_0(\mathcal{N}_{k-1, \ell}, \mathcal{M}_{k-1, k, \ell}) = 0& 
	\end{tikzcd}
	\end{equation*}
    By Lemma \ref{lem:homologyOfMjkl} and $H_1(\mathcal{N}_{k, \ell}) = 0$, we have  
    \begin{align*}
        0 \to H_1(\mathcal{N}_{k-1, \ell}) \to {\ZZ}^2 
        \xrightarrow{\partial} \oplus_{\# I^{(k)}} \ZZ \to \ZZ \to 0. 
    \end{align*}
    Here, the kernel of 
    $\partial \colon \ZZ^2 \to \oplus_{\# I^{(k)}} \ZZ$ 
    is $0$; hence, $ H_1(\mathcal{N}_{k-1, \ell}) = 0$. 
    It follows by induction that $ H_1(\mathcal{N}_{0, \ell}) = 0$ for every $\ell > 0$, and passing to the inverse limit, we have $\check{H}_1(J) = 0$ 
    by Theorem~\ref{th:cechsumi}.

    We assume $r = 1$ and $(n_1, n_2) \neq (2, 2)$ and prove the statement (1b).
    By Theorem \ref{th:H1growthInf}, we have 
    $$\liminf_{k \to \infty}\frac{1}{k}\log (\mathrm{rank} H_1(\mathcal{N}_{0, k})) \geq \log(n_1n_2 - 1).$$ 
    We now show $$\limsup_{k \to \infty}\frac{1}{k}\log (\mathrm{rank} H_1(\mathcal{N}_{0, k})) \leq \log(n_1n_2 - 1),$$
    which implies $$\lim_{k \to \infty}\frac{1}{k}\log (\mathrm{rank} H_1(\mathcal{N}_{0, k})) = \log(n_1n_2 - 1).$$
    For every $0 \leq j$ and $j+1 < \ell$, 
    by  counting the number of horizontally adjacent rectangles and vertically adjacent rectangles respectively, we have 
    \begin{align*}
    &\#\{ \{\mathbf{i}v, \mathbf{i}'v'\} \in \mathcal{N}_{j, \ell} 
    \colon \mathbf{i} \neq \mathbf{i}' \in V(\mathcal{N}_{j, j+1}) \text{ are adjacent and }  v, v'\in V(\mathcal{N}_{j+1, \ell})\} \\
    \leq 
    &\#\{ \{\mathbf{i}v, \mathbf{i}'v'\} 
    \colon \mathbf{i} \neq \mathbf{i}' \in I \text{ and }  \mathbf{i}v, \mathbf{i}'v' \text{are adjacent}\}\\
    \leq 
    &(n_1 -1)n_2^{\ell - j} + n_1^{\ell -j}(n_2 -1)\\
    \leq 
    & 2(n_1n_2-1)^{\ell - j+1}. 
    \end{align*}
    Hence, by Theorem \ref{th:countingRanks} we have
    \begin{align*}
        &\mathrm{rank} H_1(\mathcal{N}_{j, \ell}) -1\\
        \leq & (n_1n_2 - 1)\cdot (
        \mathrm{rank} H_1(\mathcal{N}_{j+1, \ell}) - 1) 
        + 2(n_1n_2-1)^{\ell - j+1}. 
    \end{align*}
    For a fixed $\ell$, as $j$ decreases, 
    we can inductively show that 
    \begin{align*}
        &\mathrm{rank} H_1(\mathcal{N}_{j, \ell}) -1\\
        \leq & (n_1n_2 - 1)^{\ell -j -1}\cdot 
        (
        \mathrm{rank} H_1(\mathcal{N}_{\ell-1, \ell}) - 1) 
        + 2(\ell -j -1)(n_1n_2 - 1)^{\ell -j + 1}.
    \end{align*}
    Since $\mathrm{rank} H_1(\mathcal{N}_{\ell-1, \ell})$ is bounded above, we deduce 
    $$\limsup_{k \to \infty}\frac{1}{k}\log (\mathrm{rank} H_1(\mathcal{N}_{0, k})) \leq \log(n_1n_2 - 1).$$ 
    This completes the proof  of the statement (1b).
    
    We assume $2 \leq r < \min\{n_{1}, n_{2}\}$ and prove statement 2. 
    The equation for $q=0$ follows from Theorem~\ref{th:H0Growth}. 
    We show the equation for $q=1$. 
    By Theorem \ref{th:H1growthInf}, we have 
    $$\liminf_{k \to \infty}\frac{1}{k}\log (\mathrm{rank} H_1(\mathcal{N}_{0, k})) \geq \log(n_1n_2 - r).$$ 
    Let $R_{j, \ell} = \mathrm{rank} H_1(\mathcal{N}_{j, \ell}) -
        \mathrm{rank} H_0(\mathcal{N}_{j, \ell})$ for every $0 \leq j < \ell$. 
    By an argument similar to that in the case $r=1$ and $(n_1, n_2) \neq (2, 2)$,
    we can show 
    \begin{align*}
    R_{j, \ell}
    \leq  (n_1n_2 - r)^{\ell -j -1}R_{\ell-1, \ell}
    + 2(\ell -j -1)(n_1n_2 - r)^{\ell -j + 1}
    \end{align*}
    for every $0 \leq j < \ell$. 
    Setting $j=0$ and $\ell = k$, we have 
    \begin{align*}
    \mathrm{rank} H_1(\mathcal{N}_{0, k})
    \leq  \mathrm{rank} H_0(\mathcal{N}_{0, k})+
    (n_1n_2 - r)^{k -1}\mathrm{rank} H_1(\mathcal{N}_{k-1, k})
    + 2(k -1)(n_1n_2 - r)^{k + 1}
    \end{align*}
    for every $k>0$.
    Since 
    $$\lim_{k \to \infty}\frac{1}{k}\log (\mathrm{rank} H_0(\mathcal{N}_{0, k})) = \log(n_1n_2 - r),$$
     we have
    $$\limsup_{k \to \infty}\frac{1}{k}\log (\mathrm{rank} H_1(\mathcal{N}_{0, k})) 
    \leq \log(n_1n_2 - r).$$
    This shows that 
    $$\lim_{k \to \infty}\frac{1}{k}\log (\mathrm{rank} H_1(\mathcal{N}_{0, k})) 
    = \log(n_1n_2 - r).$$
    This completes the proof. 
\end{proof} 

\begin{rem}\label{rem:tosionFree}
    By the proof, one can verify that 
    all the homology groups in the planar setting considered in Subsection 6.4 are torsion-free; 
    thus the $q$th homology group $H_q(\mathcal{N}_{j, k})$ is noncanonically isomorphic to the $q$th cohomology group $H^q(\mathcal{N}_{j, k})$ for every $q \geq 0$. 
    This is due to the universal coefficient theorem;
    see \cite[Corollary 5.5.4]{Spa}. 
\end{rem}

In Theorem \ref{th:randomSqDim2}, 
we observe that the exponential growth rate of the ranks of the (co)homology groups coincides with the entropy. However, the authors do not know in which generality this holds.

\section*{Declarations}
\noindent\textbf{Acknowledgment:}
The authors would like to thank Hiroki Sumi for valuable discussions, a careful reading of a previous version of the manuscript, and insightful suggestions.
The authors also thank Shigeki Akiyama for pointing them to the relevant reference \cite{Sad}.

\noindent\textbf{Funding:}
YN is partially supported by the JSPS KAKENHI Grant Number JP25K17282. 
TW is partially supported by JSPS KAKENHI 
(JP23K13000, JP24K00526, JP25K00011) 
and by JST AIP Accelerated Program JPMJCR25U6.


\noindent\textbf{Data availability:}
This manuscript has no associated data. 

\noindent\textbf{Use of Generative Al:}
The authors used ChatGPT-5.6 for English-language editing and for verification of some mathematical arguments. All mathematical statements and proofs were independently checked by the authors, who take full responsibility for the final manuscript.

\end{document}